\documentclass[11pt,a4]{amsart}
\usepackage{amsfonts,amsmath,amssymb}
\usepackage{mathabx}
\usepackage{mathrsfs}
\usepackage[latin1]{inputenc}
\usepackage[usenames,dvipsnames]{pstricks}
\newtheorem{theorem}{Theorem}[section]

\newtheorem{corollary}[theorem]{Corollary}

\newtheorem{definition}[theorem]{Definition}

\newtheorem{lemma}[theorem]{Lemma}

\newcommand{\R}{\mathbb{R}}
\newcommand{\C}{\mathbb{C}}

\newcommand{\M}{\mathrm{M}}

\newcommand{\g}{\mathrm{g}}

\newcommand{\dd}{\mathrm{d}}
\newcommand{\dv}{\,\mathrm{dv}^{n}}

\newcommand{\dvv}{\,\mathrm{dv}^{2n-1}}
\newcommand{\ds}{\,\mathrm{d\sigma}^{n-1}}

\newcommand{\dss}{\,\mathrm{d\sigma}^{2n-2}}
\newcommand{\I}{\mathcal{I}}

\newcommand{\p}{\partial}
\newcommand{\Op}{\Delta+a(x)\partial_t+q(x)}
\newcommand{\Opf}{\Delta+a_1(x)\partial_t+q_1(x)}
\newcommand{\Ops}{\Delta+a_2(x)\partial_t+q_2(x)}
\newcommand{\Opp}{\Delta-a(x)\partial_t+q(x)}

\newcommand{\norm}[1]{\|#1\|}
\newcommand{\abs}[1]{|#1|}
\newcommand{\set}[1]{\left\{#1\right\}}
\newcommand{\para}[1]{\left(#1\right)}
\newcommand{\cro}[1]{\left[#1\right]}

\newcommand{\seq}[1]{\left<#1\right>}

\newcommand{\To}{\longrightarrow}
\newcommand{\vv}{\mathrm{v}}
\newcommand{\hh}{\mathrm{h}}
\usepackage{graphicx}
\newcommand{\dive}{\textrm{div}}

\newcommand{\bea}{\begin{eqnarray}}
\newcommand{\eea}{\end{eqnarray}}
\newcommand{\beas}{\begin{eqnarray*}}
\newcommand{\eeas}{\end{eqnarray*}}
\newcommand{\bel}{\begin{equation} \label}
\newcommand{\ee}{\end{equation}}

\usepackage{times}
\allowdisplaybreaks
\begin{document}

\title[anisotropic hyperbolic equation]{Simultaneous determination of  two coefficients in the Riemannian hyperbolic equation from boundary measurements}
\author[M. Bellassoued]{Mourad~Bellassoued}
\author[Z. Rezig]{Zouhour Rezig}
\address{University of Tunis El Manar, National Engineering School of Tunis, ENIT-LAMSIN, B.P. 37, 1002 Tunis, Tunisia}
\email{mourad.bellassoued@enit.utm.tn}
\address{University of Tunis El Manar, Faculty of Sciences of Tunis, ENIT-LAMSIN, B.P. 37, 1002 Tunis, Tunisia}
\email{zouhourezig@yahoo.fr}
\date{}
\subjclass[2010]{Primary 35R30, 35L05} 
\keywords{Inverse problem, Wave equation, Riemannian manifold, Stability estimate, Dirichlet-to-Neumann map, geodesical ray transform}

\begin{abstract}
In this paper we consider the inverse problem of determining on a compact Riemannian manifold
the electric potential and the absorption coefficient in the wave equation
 with Dirichlet  data from measured Neumann boundary observations. This information is enclosed in the dynamical
Dirichlet-to-Neumann map associated to the wave equation. We prove in dimension $n\geq 2$ that the
knowledge of the Dirichlet-to-Neumann map for the wave equation uniquely determines
the absorption coefficient and the electric potential and we establish H\"older-type stability.
\end{abstract}
\maketitle
\section{Introduction and main results}
Let $(\M,\g)$ be an $n$-dimensional ($n\geq 2$) compact Riemannian manifold with smooth boundary $\p\M$ where $\g$ denotes a Riemannian metric of class $\mathcal{C}^\infty$. We let $\Delta$ denote  the Laplace-Beltrami operator on $\M$. A summary of the main Riemannian geometric notions needed in this paper is provided in Section 2. In this paper we study an inverse problem for the wave equation in the presence of an absorption coefficient and an eletric 
potential. Given $T>0$, we denote $Q=\M\times(0,T)$ and $\Sigma=\p\M\times(0,T)$. We consider the
following initial boundary value problem for the wave equation with a potential $q$ and an absorption coefficient $a$,
\begin{equation}\label{1.1}
\left\{
\begin{array}{llll}
\para{\partial^2_t-\Op}u=0  & \textrm{in }\; Q,\cr
u(\cdot,0 )=\p_tu(\cdot,0)=0 & \textrm{in }\; \M ,\cr
u=f & \textrm{on } \; \Sigma.
\end{array}
\right.
\end{equation}
Here $a,q:\M\to\R$ are real valued functions in $L^\infty(\M)$ and $f\in H^1(\Sigma)$.
\subsection{Well-posedness and direct problem}
For this paper, we use many of the notational conventions in \cite{[BellDSSF]}. Let $(\M,\g)$ be a (smooth) compact Riemannian
manifold with boundary of dimension $n \geq 2$.
We refer to \cite{[Jost]} for the differential calculus of tensor fields on Riemannian manifolds. If we fix local coordinates $x=(x^1,\ldots,x^n)$ and let
$\para{\frac{\p\, }{\p x^1},\dots,\frac{\p\,}{\p x^n}}$ denote the corresponding tangent vector fields, the inner product and the norm on the tangent
space $T_x\M$ are given by
$$
\g(X,Y)=\seq{X,Y}=\sum_{j,k=1}^n\g_{jk}X^jY^k,
$$
$$
\abs{X}=\seq{X,X}^{1/2},\qquad  X=\sum_{i=1}^nX^i\frac{\p}{\p x^i},\quad Y=\sum_{i=1}^n
Y^i\frac{\p}{\p x^i}.
$$
If $f$ is a $\mathcal{C}^1$ function on $\M$, we define the gradient of $f$ as the vector field $\nabla f$ such that
$$
X(f)=\seq{\nabla f,X}
$$
for all vector fields $X$ on $\M$. In local coordinates, we have
\begin{equation}\label{1.2}
\nabla f=\sum_{i,j=1}^n\g^{ij}\frac{\p f}{\p x^i}\frac{\p}{\p x^j} ,
\end{equation}
where $(\g^{i,j})$ is the inverse of the tensor $\g$.
The metric tensor $\g$ induces the Riemannian volume $\dv=\para{\det \g}^{1/2}\dd x^1\wedge\cdots \wedge \dd x^n$. We denote by $L^2(\M)$ the completion
of $\mathcal{C}^\infty(\M)$ with respect to the usual inner product
$$
\seq{u,v}=\int_\M u(x)\overline{v(x)} \dv,\qquad  u,v\in\mathcal{C}^\infty(\M).
$$
The Sobolev space $H^1(\M)$ is the completion of $\mathcal{C}^\infty(\M)$ with respect to the norm $\norm{\,\cdot\,}_{H^1(\M)}$,
$$
\norm{u}^2_{H^1(\M)}=\norm{u}^2_{L^2(\M)}+\norm{\nabla u}^2_{L^2(\M)}.
$$
The normal derivative is given by
\begin{equation}\label{1.3}
\p_\nu u:=\seq{\nabla u,\nu}=\sum_{j,k=1}^n\g^{jk}\nu_j\frac{\p u}{\p x^k}
\end{equation}
where $\nu$ is the unit outward vector field to $\p \M$.
Moreover, using covariant derivatives (see \cite{[Hebey]}), it is possible to define coordinate invariant norms in $H^k(\M)$, $k\geq 0$.

With these definitions in minds, we  consider the following initial boundary value problem
for the wave equation
\begin{equation}\label{1.4}
\left\{
\begin{array}{llll}
\para{\partial_t^2-\Delta+a(x)\p_t+q(x)}v(x,t)=F(x,t)  & \textrm{in }\,\,Q,\cr
v(x,0)=0,\quad\p_tv(x,0)=0 & \textrm{in }\,\,\M,\cr
v(x,t)=0 & \textrm{on } \,\, \Sigma.
\end{array}
\right.
\end{equation}
We know this problem is well-posed, since we have the following existence and uniqueness result, see \cite{[KKL]}.
\begin{lemma}\label{L.1.1}
Let $T>0$, $a\in L^\infty(\M)$, and $q\in L^\infty(\M)$. Assuming $F\in H^1(0,T;L^2(\M))$ such that $F(\cdot,0)=0$ in $\M$, then there exists a unique solution
$v$ to \eqref{1.4} such that
 $$
 v\in \mathcal{C}^2(0,T;L^2(\M))\cap \mathcal{C}^1(0,T;H^1_0(\M))\cap \mathcal{C}(0,T; H^2(\M)).
 $$ 
 Furthermore, there is a constant $C>0$ such that
\begin{equation}\label{1.5}
\norm{\p_tv(\cdot,t)}_{L^2(\M)}+\norm{\nabla v(\cdot,t)}_{L^2(\M)}\leq
C\norm{F}_{L^2(Q)},
\end{equation}
and
\begin{equation}\label{1.6}
\norm{\p_t^2 v(\cdot,t)}_{L^2(\M)}+\norm{\nabla\p_t v(\cdot,t)}_{L^2(\M)}+\norm{\Delta v(\cdot,t)}_{L^2(\M)}\leq C\norm{F}_{H^1(0,T;L^2(\M))}.
\end{equation}
\end{lemma}
A proof of the following lemma may be found for instance in \cite{[Ika]}.
\begin{lemma}\label{L.1.2}
Let $f\in H^1(\Sigma)$ be a function such that $f(x,0)=0$ for all $x\in\p\M$.
There exists an unique solution
\begin{equation}\label{1.7}
u\in \mathcal{C}^1(0,T;L^2(\M))\cap \mathcal{C}(0,T;H^1(\M))
\end{equation}
to the problem (\ref{1.1}).  Furthermore, there is a
constant $C>0$ such that
\begin{equation}\label{1.8}
\norm{\p_\nu u}_{L^2(\Sigma)}\leq C\norm{f}_{H^1(\Sigma)}.
\end{equation}
\end{lemma}
\subsection{Inverse problem and Main result}
From the physical viewpoint, our inverse problem consists in
determining the properties (e.g. an absorption coefficient) of an
inhomogeneous medium by probing it with disturbances generated on
the boundary. The measurements are responses of the medium to these
disturbances which are measured on the boundary, and the
goal is to recover the potential $q(x)$ and the absorption coefficient $a(x)$ which describes the property of the
medium. Here we assume that the medium is quiet initially, and $f$
is a disturbance which is used to probe the medium. Roughly
speaking, the data is $\p_\nu u$ measured on the boundary for
different choices of $f$.
\medskip

 We may define the Dirichlet to Neumann (D-N) map associated with hyperbolic problem (\ref{1.1}) by
\begin{equation}\label{1.9}
\Lambda_{a,q}(f)=\p_\nu u,\quad f\in \mathcal{H}_0^{1}(\Sigma)=\set{f\in H^1(\Sigma),\,\, f(\cdot,0)=0\,\textrm{on}\,\,\p\M}.
\end{equation}
Therefore the Dirichlet-to-Neumann map  $\Lambda_{a,q}$ defined by (\ref{1.9}) is continuous. We denote by $\norm{\Lambda_{a,q}}$ its norm in $ \mathscr{L}\para{\mathcal{H}_0^1(\Sigma);L^2(\Sigma)}$.
\smallskip

For a Riemannian manifold $(\M,\g)$ with boundary $\p\M$, we denote by $\nabla$ the Levi-Civita connection on $(\M,\g)$. For a point $x \in \p\M$, the second quadratic form of the boundary
$$
\Pi(\xi,\xi)=\seq{\nabla_\xi\nu,\xi},\quad \xi\in T_x(\p\M),
$$
is defined on the space $T_x(\p\M)$. We say that the boundary is strictly convex if the form is positive-definite for all $x \in \p\M$ (see \cite{[Sh]}).
\begin{definition}
We say that the Riemannian manifold $(\M,\g)$ (or that the metric $\g$) is simple in $\M$, if $\p \M$ is strictly convex with respect to $\g$, and for any $x\in \M$, the exponential map $\exp_x:\exp_x^{-1}(\M)\To \M$ is a diffeomorphism. The latter means that every two points $x; y \in \M$ are joined by a unique geodesic smoothly depending on $x$ and $y$.
\end{definition}
Note that if $(\M,\g)$ is simple, one can extend it to a simple manifold $\M_{1}$ such that $\M_1^{\textrm{int}}\supset\M$.
\smallskip

Let us now introduce the admissible sets of absorption coefficients $a$ and electric potentials $q$. Let $m_1,m_2>0$ and $\eta> n/2$ be given, set
\begin{equation}\label{1.10}
\mathscr{A}(m_1,\eta)= \set{a\in W^{2,\infty}(\M),\,\,\norm{a}_{H^\eta(\M)}\leq m_1},
\end{equation}
and 
\begin{equation}\label{1.11}
\mathscr{Q}(m_2)= \set{q\in W^{2,\infty}(\M),\,\,\norm{q}_{H^2(\M)}\leq m_2}.
\end{equation}
Introduce one more notation. Given $x\in\M$ and a 2-plane $\pi\subset T_x\M$, denote by $K(x,\pi)$ the sectional curvature of $\pi$ at $x$. For $\xi\in T_x\M$ with $\abs{\xi}=1$, put
$$
K(x,\xi)=\sup_{\pi,\,\xi\in\pi}\,K(x,\pi),\quad K^+(x,\xi)=\max\{0,K(x,\xi)\}.
$$
Define the following characteristic:
$$
k^+(\M,\g)=\sup_\gamma\int_{\ell_1}^{\ell_2}tK^+(\gamma(t),\dot{\gamma}(t))dt,
$$
where $\gamma\,:\,[\ell_1,\ell_2]\to\M$, ranges in the set of all unit speed geodesic in $\M$.\\
The main results of this paper are as follows.
\begin{theorem}\label{Th2}
Let $(\M,\g)$ be a simple compact Riemannian manifold with boundary of dimension $n \geq 2$ such that $k^+(\M,\g)<1$, and let $T>\textrm{Diam}_\g(\M)$.
There exist  $C>0$ and $\gamma\in(0,1)$ such that for any $a_1,a_2\in\mathscr{A}(m_1,\eta)$ and $q_1,q_2\in\mathscr{Q}(m_2)$ coincide near the boundary $\p\M$, the following estimate holds true
\begin{equation}\label{1.12}
\norm{a_1-a_2}_{L^2(\M)}+\norm{q_1-q_2}_{L^2(\M)}\leq C\norm{\Lambda_{a_1,q_1}-\Lambda_{a_2,q_2}}^\gamma
\end{equation}
where $C$ depends on $\M$, $m_1,m_2$, $\eta$ and $n$.
\end{theorem}
By Theorem \ref{Th2}, we can readily derive the following uniqueness result
\begin{corollary}
Let $(\M,\g)$ be a simple compact Riemannian manifold with boundary of dimension $n \geq 2$, such that $k^+(\M,\g)<1$ and let $T>\textrm{Diam}_\g(\M)$, we have that $\Lambda_{a_1,q_1}=\Lambda_{a_2,q_2}$ implies
$a_1 =a_2$ and $q_1=q_2$ almost everywhere in $\M$.
\end{corollary}
\subsection{Relation to the literature}
In recent years significant progress has been made for the problem of identifying one coefficient in the euclidean hyperbolic equation ($\g_{ij}=\delta_{ij}$). In \cite{[Rakesh-Symes]}, Rakesh and Symes prove that the D-to-N map determines uniquely the time-independent potential in a wave equation. Ramm and Sj\"ostrand \cite{[Ramm-Sjostrand]} has extended the result in \cite{[Rakesh-Symes]} to the case of time-dependent potentials.
Isakov \cite{[Isakov1]} has considered the simultaneous uniqueness determination of a zeroth order coefficient and an absorption coefficient.  A key ingredient in the existing results is the construction of complex geometric optics solutions of the wave equation, concentrated along a line, and the relationship between the hyperbolic D-to-N map and the $X$-ray transform play a crucial role. In \cite{[Pestov]} Pestov propose a linear procedure based on the boundary control method for determining both coefficients, absorbtion and speed, for the wave equation. 
\smallskip

For the stability estimates, Sun \cite{[Sun]} established in the Euclidean case stability estimates for potentials from the Dirichlet-to-Neumann
map. In \cite{[BCY]} the authors consider the stability in an inverse problem of determining the potential $q$ entering the wave equation in a bounded smooth domain of $\R^d$ from boundary observations. The observation is given by a hyperbolic (dynamic) Dirichlet to Neumann map associated to a wave equation and prove a $\log$-type stability estimate in determining $q$ from a partial Dirichlet to Neumann map. For the wave equation with a lower order term $q(t, x)$, Waters \cite{[Waters]} proves that we can recover the $X$-ray transform of time dependent potentials $q(t, x)$ from
the dynamical Dirichlet-to-Neumann map in a stable way. He derive
conditional H\"older stability estimates for the $X$-ray transform of $q(t, x)$.
\smallskip

In the case of Riemannian wave equation, Bellassoued and Dos Santos Ferriera \cite{[BellDSSF]}  seek stability estimates in the inverse problem of determining the potential or the velocity in a wave equation posed in a simple riemannian wave equation $(\M,\g)$ from measured Neumann boundary observations. The authors prove in dimension $n\geq 2$ that the knowledge of the
Dirichlet-to-Neumann map for the wave equation uniquely determines the
electric potential and  they show a Hölder-type stability in determining the potential. Similar results for the determination of velocities close
to $1$ is also given.
\smallskip

In \cite{[SU]} and \cite{[SU2]} Stefanov and Uhlmann considered the inverse problem of determining a Riemannian metric $\g$
on a Riemannian manifold $(\M,\g)$ with boundary from the hyperbolic Dirichlet-to-Neumann map associated to solutions of the wave equation $(\p_t^2-\Delta_\g)u=0$. A H\"older type of conditional stability estimate was proven in \cite{[SU]} for metrics close enough to the Euclidean metric in $\mathcal{C}^k$, $k\geq 1$ or for generic simple metrics in \cite{[SU2]}. It is clear that one cannot hope to uniquely determine the metric
$\g=(\g_{jk})$ from the knowledge of the Dirichlet-to-Neumann map
$\Lambda_{\g,a,\,q}$. As was noted in \cite{[SU]}, the Dirichlet-to-Neumann map
is invariant under a gauge transformation of the metric~$\g$. Namely, given a diffeomorphism $\psi:\M\to \M$ such that
$\psi|_{\p\M}={\rm Id}$ one has $\Lambda_{\psi^*\g,a,\,q}=\Lambda_{\g,a,\,q}$
where $\psi^*\g$ denotes the pullback of the metric $\g$ under $\psi$.
\smallskip

In \cite{[Montalo]}, Montalto studies the stability of simultaneously recovering the Riemannian metric $\g$, a covector field $b$ and a potential $q$ in a Riemannian manifold $\M$ from the boundary measurements modeled by the Dirichlet-to-Neumann map. He shows that, assuming the metric is close to a generic simple metric, and the two covector close, a conditional H\"older-type stability for the recovery holds up to the natural gauge transformations that fix the boundary. This result generalizes the results in \cite{[BellDSSF]} and \cite{[SU]}.
\smallskip

In \cite{[BK]} Belishev and Kurylev gave an affirmative answer  to the general problem of finding a smooth metric from the Dirichlet-to-Neumann map.
Their approach is based on the boundary control method introduced by Belishev \cite{[1]} and uses in an essential way an unique continuation
property. Unfortunately it seems unlikely that this method would provide stability estimates even under geometric and topological restrictions.
Their method also solves the problem of recovering $\g$ through boundary spectral data.
The boundary control method  gave rise to several refinements of the results of \cite{[BK]}: one can cite for instance \cite{[KL]}, \cite{[KKL]}
and \cite{[AKKLT]}.  
\smallskip

The importance of control theory for inverse problems was first understood by Belishev \cite{[1]}. He used control theory to develop the first variant of the boundary control (BC) method. Later, the idea based on control theory was combined with the geometrical ones. The
importance of the geometry for inverse problems follows the fact that any elliptic second-order differential operator gives rise to a
Riemannian metric in the corresponding domain. The role of this metric becomes clearer if we consider the solutions of the corresponding
wave equation. Indeed, these waves propagate with the unit speed along geodesics of this Riemannian metric. These geometric ideas where
introduced to the boundary control method in \cite{[KL]}, \cite{[KKL]}.
\smallskip

In this paper, the inverse problem under consideration is whether, for a fixed metric $\g$, the knowledge of the
Dirichlet-to-Neumann map $\Lambda_{\g,a,\,q}$ on the boundary
 uniquely determines the electric potential $q$ and the absorption coefficient $a$.
 \smallskip

Uniqueness properties for local Dirichlet-to-Neumann maps associated with the wave equation are rather well understood (e.g.,
Belishev \cite{[1]}, Katchlov, Kurylev and Lassas \cite{[KKL]},
Kurylev and Lassas \cite{[KL]}) but stability for such operators is far from being apprehended.
For instance, one may refer to Isakov and Sun \cite{[IS]} where a local Dirichet-to-Neumann map yields a
stability result in determining a coefficient in a subdomain.
There are quite a few works on Dirichlet-to-Neumann maps, so our references
are far from being complete: see also  Eskin \cite{[Eskin1]}-\cite{[Eskin2]},  Uhlmann \cite{[Uhlmann]} as related papers.
\smallskip

The main goal of this paper is to study the stability of the inverse problem for the dynamical anisotropic wave equation.
The approach that we develop is a dynamical approach. It is based on the consideration of the wave equation and involves various techniques to study an initial-boundary value problem for the hyperbolic equation. In this paper we prove a
H\"older-type estimate which shows that a dispersion term $q$ and the absorption coefficient $a$ depends stably on the Dirichlet-to-Neumann map. Our approach here is different from \cite{[Montalo]} in order to prove a stability estimate without a smallness assumpltion for the coefficients. The main idea is to probe the medium by real geometric optics solutions of the wave equation, concentrated along a geodesic line, starting
on one side of the boundary, and measure responses of the medium on other side of the boundary and using directely a stability estimate for the geodesic $X$-ray tranform without passing by the normal operator as in \cite{[Montalo]}.
\smallskip

The outline of the paper is as follows. In section 2 we give an important stability estimate for the geodiscal ray transfom. In section 3 we construct special geometrical optics solutions to wave equations with potential and absorption coefficients. In section 4 and 5, we establish stability estimates for the absorption coefficient and the electric potential. The appendix A is devoted to the study of the Poisson kernel in the tangent sphere bundle. 
\section{Stability estimate for the geodesical $X$-ray transform}
\subsection{Geodesical ray transform on a simple manifold}
\setcounter{equation}{0}
The geodesic $X$-ray transform of a function is defined by integrating over geodesics.
It is naturally arises in linearization of the problem of determining a coefficient in partial differential equation. The $X$-ray transform also
arises in Computer Tomography, Positron Emission Tomography, geophysical imaging
in determining the inner structure of the Earth, ultrasound imaging. Uniqueness
result and stability estimates of the geodesic $X$-ray transform were obtained by Mukhometov \cite{[Mukh]} for simple surface. For simple manifolds of any dimension this
result was proven in \cite{[SU2]}, \cite{[SU3]}, see also V. A. Sharafutdinov's book \cite{[Sh]}. In his paper Dairbekov generalized this result for nontrapping manifolds without conjugate points \cite{[Dai]}. Fredholm type inversion formulas were given in \cite{[PU]} by Pestov and Uhlmann.
\medskip

In this section we first collect some  formulas needed in the rest of this paper and introduce the geodesical $X$-ray transform on the manifolds we will be using. Let $(\M,\g)$ be a Riemannian manifold, for $x\in \M$ and $\xi\in T_x\M$ we let  $\gamma_{x,\xi}$ denote the unique geodesic starting at the point $x$ in the direction $\xi$.  By
\begin{align*}
S\M=\set{(x,\xi)\in T\M;\,\abs{\xi}=1}, \quad
S^*\M=\set{(x,p)\in T^*\M;\,\abs{p}=1},
\end{align*}
we denote the sphere bundle and co-sphere bundle of $\M$. The exponential map $\exp_x:T_x\M\to \M$ is given by
\begin{equation}\label{2.1}
\exp_x(v)=\gamma_{x,\xi}(\abs{v}),\quad \xi=\frac{v\,\,}{\abs{v}}.
\end{equation}
A compact Riemannian manifold $(\M,\, \g)$ with boundary is called a convex non-trapping
manifold, if it satisfies two conditions:
\begin{enumerate}
    \item[(i)] the boundary $\p \M$ is strictly convex, i.e., the second fundamental form of the boundary is positive definite at every
    boundary point,
    \item[(ii)] all geodesics having finite length in $\M$, i.e., for each $(x,\xi) \in S\M$, the maximal
    geodesic $\gamma_{x,\xi}(t)$ satisfying the initial conditions $\gamma_{x,\xi}(0) = x$ and $\dot{\gamma}_{x,\xi}(0) = \xi$ is defined on
    a finite segment with extremities $\ell_{-}(x,\xi)$ and $\ell_{+}(x,\xi)$. We recall that a geodesic $\gamma: [a, b] \To M$ is maximal
    if it cannot be extended to a segment $[a-\varepsilon_1, b+\varepsilon_2]$, where $\varepsilon_i \geq 0$ and $\varepsilon_1 + \varepsilon_2 > 0$.
\end{enumerate}
An important subclass of convex non-trapping manifolds are simple manifolds. We say that a compact Riemannian
manifold $(\M, \g)$ is simple if it satisfies the following properties
\begin{enumerate}
     \item[(a)] $(\M,\g)$ is convex and non-trapping, 
     \item[(b)] there are no conjugate points on any geodesic.
\end{enumerate}
A simple $n$-dimensional Riemannian manifold is diffeomorphic to a closed ball in $\R^n$, and any pair of
points in the manifold are joined by an unique geodesic.
\medskip

Let $(x,\xi)\in S\M$, there exist a unique geodesic $\gamma_{x,\xi}$ associated to $(x,\xi)$ which is maxmimally defined on a finite intervall $[\ell_-(x,\xi),\ell_+(x,\xi)]$, with $\gamma_{x,\xi}(\ell_\pm(x,\xi))\in\p\M$. We define the geodesic flow $\phi_t$ as following
\begin{equation}\label{2.2}
\phi_t:S\M\to S\M,\quad \phi_t(x,\xi)=(\gamma_{x,\xi}(t),\dot{\gamma}_{x,\xi}(t)),\quad t\in [\ell_-(x,\xi),\ell_+(x,\xi)],
\end{equation}
and $\phi_t$ is a flow, that is, $\phi_t\circ\phi_s=\phi_{t+s}$.
\smallskip

Now, we introduce the submanifolds of inner and outer vectors of $S\M$
\begin{equation}\label{2.3}
\p_{\pm}S\M =\set{(x,\xi)\in S\M,\, x \in \p \M,\, \pm\seq{\xi,\nu(x)}< 0},
\end{equation}
where $\nu$ is the unit outer normal to the boundary. Note that the manifolds $\p_+ S\M$ and $\p_-S\M$ have the same boundary $S(\p \M)$, and $\p S\M = \p_+ S\M\cup \p_- S\M$. We denote by  $\mathcal{C}^\infty(\p_+ S\M)$ be space of smooth functions on the manifold $\p_+S\M$. Thus we can define two functions $\ell_\pm:S\M\to\R$ which satisfy
$$
\ell_-(x,\xi)\leq 0,\quad \ell_+(x,\xi)\geq 0,\quad \ell_+(x,\xi)=-\ell_-(x,-\xi),
$$
$$
\ell_-(x,\xi)=0,\quad (x,\xi)\in\p_+S\M,\quad \ell_+(x,\xi)=0,\quad (x,\xi)\in\p_-S\M,
$$
$$
\ell_-(\phi_t(x,\xi))=\ell_-(x,\xi)-t,\quad \ell_+(\phi_t(x,\xi))=\ell_+(x,\xi)-t.
$$
For $(x,\xi)\in\p_+ S\M$,  we denote by $\gamma_{x,\xi} : [0,\ell_+(x,\xi)] \to \M$ the maximal
geodesic satisfying the initial conditions $\gamma_{x,\xi}(0) = x$ and $\dot{\gamma}_{x,\xi}(0) = \xi$.
\smallskip

Concerning smoothness properties of  $\ell_{\pm}(x,\xi)$, we can see that these functions are smooth near a point $(x,\xi)$ such that the geodesic $\gamma_{x,\xi}(t)$ intersects $\p\M$ transversely for $t=\ell_{\pm}(x,\xi)$. By strict convexity of $\p\M$, the functions $\ell_{\pm}(x,\xi)$ are smooth on $T\M \setminus T(\p\M).$ In fact, all points of $T\M \cap T(\p\M)$ are singular for $\ell_{\pm}$; since some derivatives of these functions are unbounded in a neighbourhood of such points. In particular, $\ell_{+}$ is smooth on $\p_+SM$, see Lemma 4.1.1 of \cite{[Sh]}.
\smallskip

The Riemannian scalar product on $T_x\M$ induces the volume form on $S_x\M$,
denoted by $\dd \omega_x(\xi)$ and given by
$$
\dd \omega_x(\xi)=\sqrt{\abs{\g}} \, \sum_{k=1}^n(-1)^k\xi^k \dd \xi^1\wedge\cdots\wedge \widehat{\dd \xi^k}\wedge\cdots\wedge \dd \xi^n.
$$
As usual, the notation $\, \widehat{\cdot} \,$ means that the corresponding factor has been dropped.
We introduce the volume form $\dvv$ on the manifold $S\M$ by
$$
\dvv (x,\xi)=\dd\omega_x(\xi)\wedge \dv,
$$
where $\dv$ is the Riemannnian volume form on $\M$. By Liouville's theorem, the form $\dvv$ is preserved by the geodesic flow. The
corresponding volume form on the boundary $\p S\M =\set{(x,\xi)\in S\M,\, x\in\p \M}$ is given
by
$$
\dss=\dd\omega_x(\xi) \wedge \ds,
$$
where $\ds$ is the volume form of $\p \M$.
\smallskip

Let $L^2_\mu(\p_+S\M)$ be the space of square integrable functions with respect to the measure $\mu(x,\xi)\dss$ with
$\mu(x,\xi)=\abs{\seq{\xi,\nu(x)}}$. This Hilbert space is endowed with the scalar product
\begin{equation}\label{2.4}
\para{u,v}_\mu=\int_{\p_+S\M}u(x,\xi) \overline{v}(x,\xi) \mu(x,\xi)\dss.
\end{equation}
The ray transform (also called geodesic $X$-ray transform) on a convex non trapping manifold $\M$ is the linear operator
\begin{equation}\label{2.5}
\I:\mathcal{C}^\infty(\M)\To \mathcal{C}^\infty(\p_+S\M),
\end{equation}
defined by the equality
\begin{equation}\label{2.6}
\I f(x,\xi)=\int_0^{\ell_+(x,\xi)}f(\gamma_{x,\xi}(t))\, \dd t,\quad (x,\xi)\in\p_+S\M.
\end{equation}
The right-hand side of (\ref{2.6}) is a smooth function on $\p_+S\M$ because the integration limit $\ell_+(x,\xi)$ is a smooth function on $\p_+S\M$.
The ray transform on a convex non trapping manifold $\M$ can be extended as a bounded operator
\begin{equation}\label{2.7}
\I:H^k(\M)\To H^k(\p_+S\M),
\end{equation}
for every integer $k\geq 1$, see Theorem 4.2.1 of \cite{[Sh]}.
\medskip

\subsection{Inverse inequality for the geodesical ray-transform}
This subsection concerns the problem of inverting the ray transform.
\smallskip

Let $R$ be the curvature tensor of the Levi-Civita connection $\nabla_X$ defined by
$$
R(X,Y)Z=\nabla_X\nabla_Y Z+\nabla_Y\nabla_X Z-\nabla_{[X,Y]} Z, \quad X,Y,Z\in T\M.
$$
For a point $x \in \M$ and a two-dimensional subspace $\pi \subset T_x\M,$ the number
$$
K(x,\pi)=\frac{\langle R(\xi,\eta)\eta,\xi\rangle}{\vert \xi \vert^2 \vert \eta \vert^2-\langle \xi, \eta \rangle^2}, 
$$ 
is independent of the choice of the basis $\xi$, $\eta$ for $\pi$. It is called the sectional curvature of the manifold $\M$ at the point $x$ and in the two-dimentional direction $\pi$.\\
For $(x,\xi)\in T\M$, we set $K(x,\xi)=\underset{\pi \ni \xi}{\sup} K(x,\pi) $ and
\begin{equation}\label{2.8}
 K^+(x,\xi)=\max\{0,K(x,\xi)\}. 
\end{equation}
 For a simple compact Riemannian manifold $(\M,\g)$, we set 
 \begin{equation}\label{2.9}
 k^+(\M,\g)=\sup\lbrace \int_0^{\ell_+(x,\xi)} tK^+(\gamma_{x,\xi}(t),\dot{\gamma}_{x,\xi}(t)) dt, \; (x,\xi)\in \p_+S\M\rbrace.
 \end{equation}
The aim of this section is to prove the following theorem.
\begin{theorem}\label{theorem}
For every simple compact Riemannian manifold $(\M,\g)$ with $k^+(\M,\g)<1$, there exist $C>0$ such that the following stability estimate 
\begin{equation}\label{2.10}
 \Vert f \Vert _{L^2(\M)}\leq C \Vert \I f \Vert_{H^1(\p_+S\M)}, 
\end{equation} 
holds for any $f\in H^1(\M)$.
\end{theorem}
By density arguments, it suffices to prove the theorem for $f \in \mathcal{C}^{\infty}(\M)$. Indeed, if $f\in H^1(M)$, then we can find  a sequence $(f_k)_k$ in $\mathcal{C}^{\infty}(\M)$ converging towards $f$ in $H^1(\M).$
 The ray transform $\I$ is a bounded operator from $H^1(\M)$ into $H^1(\p_+S\M)$, so the sequence $(\I f_k)_k$ converges towards $\I f$ in $H^1(\p_+S\M).$ Applying the theorem for $f_k$ and passing to the limit as $k \rightarrow +\infty,$ we deduce that $ \Vert f \Vert _{L^2(\M)}\leq C \Vert \I f \Vert_{H^1(\p_+S\M)}$.
  \smallskip
  
   Before starting the proof of the theorem, we need to introduce some notions and notations. We will use the Einstein summation convention to abbreviate the notations. When an index variable appears twice in a simple term and is not otherwise defined, it implies summation of that term over all the values of the index. For example, for $i\in{1,2,3}$, $c_ix^i$ means $c_1x^1+c_2x^2+c_3x^3$.
\smallskip

Let $\tau_s^r\M$ be the bundle of tensors of degree $(r,s)$ on $\M$. Its sections denoted by $T_s^r$ are called \textit{tensor fields of degree $(r,s)$.} Let $U$ be a domain of $\M
$.
 We denote by $\mathcal{C}^{\infty}(\tau_s^r\M,U)$ the $\mathcal{C}^{\infty}(U)$-module of smooth sections of the bundle $\tau_s^r \M$ over $U$. The notation $\mathcal{C}^{\infty}(\tau_s^r\M,\M)$ will usually be abbreviated to $\mathcal{C}^{\infty}(\tau_s^r\M).$
  If $(x^1,\dots,x^n)$ is a local coordinate system defined in a domain $U$, then any tensor field $u\in \mathcal{C}^{\infty}(\tau_s^r \M,U)$ can be uniquely represented as
\begin{equation}\label{2.11}
u=u_{j_1,\dots,j_s}^{i_1,\dots,i_r} \frac{\p}{\p x^{i_1}}\otimes\dots\otimes\frac{\p}{\p x^{i_r}}\otimes dx^{j_1} \otimes \dots \otimes dx^{j_s},
\end{equation}
where $u_{j_1,\cdots,j_s}^{i_1,\cdots,i_r} \in \mathcal{C}^{\infty}(U)$ are called the coordinates of the field $u$ in the given coordinate system. We will usually abbreviate (\ref{2.11}) as follows
\begin{equation}\label{2.12}
 u=(u_{j_1,\cdots,j_s}^{i_1,\cdots,i_r}).
\end{equation}
The bundles $\tau_s^r \M$ and $\tau_r^s \M$ are dual to each other and, consequently, $\mathcal{C}^{\infty}(\tau_s^r \M)$ and $\mathcal{C}^{\infty}(\tau_r^s \M)$ are mutually dual $\mathcal{C}^{\infty}(\M)$-modules. This implies in particular that a field $u\in \mathcal{C}^{\infty}(\tau_s^1 \M)$ can be considered as a $\mathcal{C}^{\infty}(\M)$-multilinear mapping $u:\mathcal{C}^{\infty}(\tau_0^1 \M)\times\dots\times\mathcal{C}^{\infty}(\tau_0^1 \M)\longrightarrow \mathcal{C}^{\infty}(\tau_0^1 \M).$ Consequently, a given connection $\nabla$ on $\M$ defines the $\C$-linear mapping (denoted by the same letter)
\begin{equation}\label{2.13}
\nabla:\mathcal{C}^{\infty}(\tau_0^1 \M)\longrightarrow \mathcal{C}^{\infty}(\tau_1^1 \M)
\end{equation}  by the formula $(\nabla v)(u)=\nabla_uv,$ 
The tensor field $\nabla v$ is called the \textit{covariant derivative} of the vector field $v$ (with respect to the given connection). The covariant differenciation defined on vector fields can be transferred to tensor fields ( see \cite{[Sh]} Theorem 3.2.1 pp. 85)
\begin{equation}\label{2.14}
\nabla:\mathcal{C}^{\infty}(\tau_s^r \M)\longrightarrow \mathcal{C}^{\infty}(\tau_{s+1}^r \M)
\end{equation}
such that (\ref{2.14}) coincides with the mapping (\ref{2.13}), for $r=1$ and $s=0$ and that for a field $$u=u_{j_1,\dots,j_s}^{i_1,\dots,i_r} \frac{\p}{\p x^{i_1}}\otimes\dots\otimes\frac{\p}{\p x^{i_r}}\otimes dx^{j_1} \otimes \dots \otimes dx^{j_s},$$ the field $\nabla u$ is defined by $$\nabla u=\nabla_ku_{j_1,\dots,j_s}^{i_1,\dots,i_r}  \frac{\p}{\p x^{i_1}}\otimes\dots\otimes\frac{\p}{\p x^{i_r}}\otimes dx^{j_1} \otimes \dots \otimes dx^{j_s}\otimes dx^k,$$ where 
\begin{equation}\label{2.15}
\nabla_ku_{j_1\dots j_s}^{i_1\dots i_r} =\frac{\p}{\p x^k}u_{j_1\dots j_s}^{i_1 \dots i_r}+\sum_{m=1}^r\Gamma^{i_m}_{kp}u_{j_1 \dots j_s}^{i_1 \dots i_{m-1} p i_{m+1} \dots i_r}-\sum_{m=1}^s\Gamma^{p}_{kj_m}u_{j_1 \dots j_{m-1} p j_{m+1} \dots j_s}^{i_1 \dots  i_r},
\end{equation}
where $\Gamma_{kq}^p$ is the Christoffel symbol.

Now, we will extend this covariant differenciation for tensors on $T\M$. If $(x^1,\dots, x^n)$ is a local coordinate system  defined in a domain $U \subset \M$, then we denote by $\frac{\p}{\p x^i}$ the coordinate vector fields and by $dx^i$ the coordinate covector fields. We recall that the coordinates of a vector $\xi \in T_x\M$ are the coefficients of the expansion $\xi=\xi^i \frac{\p}{\p x^i}$. Let $p$ be the projection on $\M$. On the domain $p^{-1}(U)\subset T\M$, the family of the functions $(x^1, \dots,x^n,\xi^1,\dots,\xi^n)$ is a local coordinate system which is called \textit{associated} with the system $(x^1,\dots,x^n).$ A local coordinate system  on $T\M$ will be called a \textit{natural coordinate system} if it is associated with some local coordinate system on $\M$. In the sequel, we will use only such coordinate systems on $T\M$. The algebra of tensor fields of the manifold $T\M$ is generated locally by the coordinate fields $\frac{\p}{\p x^i},\frac{\p}{\p \xi^i}, dx^i, d\xi^i.$ A tensor $u \in T^r_{s,(x,\xi)}(T\M)$ of degree $(r,s)$ at a point $(x,\xi)\in T\M$ is called \textit{semibasic} if in some (and so, in any) natural coordinate system, it can be represented as:
\begin{equation}\label{2.16}
u=u_{j_1 \dots j_s}^{i_1 \dots i_r} \frac{\p}{\p \xi^{i_1}}\otimes\dots\otimes\frac{\p}{\p \xi^{i_r}}\otimes dx^{j_1} \otimes \dots \otimes dx^{j_s}.
\end{equation}
We will abbreviate this equality to $$u=(u_{j_1 \dots j_s}^{i_1 \dots i_r} ).$$ We denote by $\beta_s^r\M$ the subbundle in $\tau_s^r(T\M)$ made of all semibasic tensors of degree $(r,s).$ Note that $\mathcal{C}^{\infty}(\beta_0^0 \M)=\mathcal{C}^{\infty}(T \M)$. The elements of $\mathcal{C}^{\infty}(\beta_0^1 \M)$ are called the \textit{semibasic vector fields}, and the elements of $\mathcal{C}^{\infty}(\beta_1^0 \M)$ are called \textit{semibasic covector fields}.
Tensor fields on $M$ can be identified with the semibasic tensor fields on $T\M$ whose components are independent of the second argument $\xi$. Thus we obtain the canonical imbedding
\begin{equation}\label{2.17}
\iota:\mathcal{C}^{\infty}(\tau_s^r \M) \subset \mathcal{C}^{\infty}(\beta_s^r \M).
\end{equation}
Note that $\iota( \frac{\p}{\p x^{i}})= \frac{\p}{\p \xi^{i}}$ and $\iota(dx^i)=dx^i.$
\smallskip

For $u\in \mathcal{C}^{\infty}(\beta_s^r \M)$, we define two semibasic tensor fields $\overset{\vv}{\nabla}u$ and $\overset{\hh}{\nabla}u$ by the formulas 
\begin{equation}\label{2.18}
\overset{\vv}{\nabla}u =\overset{\vv}{\nabla}_k u_{j_1 \dots j_s}^{i_1 \dots i_r} \frac{\p}{\p \xi^{i_1}}\otimes\dots\otimes\frac{\p}{\p \xi^{i_r}}\otimes dx^{j_1} \otimes \dots \otimes dx^{j_s}\otimes dx^k,
\end{equation}
where
\begin{equation}\label{2.19}
\overset{\vv}{\nabla}_k u_{j_1 \dots j_s}^{i_1 \dots i_r}=\frac{\p }{\p \xi^k} u_{j_1 \dots j_s}^{i_1 \dots i_r},
\end{equation}
and 
\begin{equation}\label{2.20}
\overset{\hh}{\nabla}u =\overset{\hh}{\nabla}_k u_{j_1 \dots j_s}^{i_1 \dots i_r} \frac{\p}{\p \xi^{i_1}}\otimes\dots\otimes\frac{\p}{\p \xi^{i_r}}\otimes dx^{j_1} \otimes \dots \otimes dx^{j_s}\otimes dx^k,
\end{equation}
where
\begin{multline}\label{2.21}
\overset{\hh}{\nabla}_ku_{j_1\dots j_s}^{i_1\dots i_r} =\frac{\p}{\p x^k}u_{j_1\dots j_s}^{i_1 \dots i_r}-\Gamma^{p}_{kq}\xi^q \frac{\p}{\p \xi^{p}}u_{j_1\dots j_s}^{i_1\dots i_r}\cr
+\sum_{m=1}^r\Gamma^{i_m}_{kp}u_{j_1 \dots j_s}^{i_1 \dots i_{m-1} p i_{m+1} \dots i_r}-\sum_{m=1}^s\Gamma^{p}_{kj_m}u_{j_1 \dots j_{m-1} p j_{m+1} \dots j_s}^{i_1 \dots  i_r}.
\end{multline}
\smallskip

We thus obtain two well-defined differential operators $\overset{\vv}{\nabla}, \overset{\hh}{\nabla}:\mathcal{C}^{\infty}(\beta_s^r \M)\longrightarrow \mathcal{C}^{\infty}(\beta_{s+1}^r \M)$ that are respectively called the \textit{vertical} and \textit{horizontal} covariant derivatives.\\
For $u\in\mathcal{C}^\infty(T\M)$ the covariant field $\overset{\hh}{\nabla} u\in\mathcal{C}^\infty(T^*\M)$ given in a coordinates system by
\begin{equation}\label{2.22}
\overset{\hh}{\nabla} u=(\overset{\hh}{\nabla}_k u)dx^k,\quad \overset{\hh}{\nabla}_k u=\frac{\p u}{\p x^k}-\Gamma_{kq}^p\xi^q\frac{\p u}{\p\xi^p}.
\end{equation}
The vertical covariant derivative of $u$ is given by
\begin{equation}\label{2.23}
\overset{\vv}{\nabla}u=(\overset{\vv}{\nabla}_ku)dx^k,\quad \overset{\vv}{\nabla}_ku=\frac{\p u}{\p\xi^k}.
\end{equation}
We can show that these derivatives satisfy the following commutation formulas (for more details, see \cite{[Sh]}, pp. 95).
\begin{equation}\label{2.24}
\overset{\vv}{\nabla}_k \overset{\hh}{\nabla}_l=\overset{\hh}{\nabla}_l\overset{\vv}{\nabla}_k.
\end{equation} 
We can also prove the following relations
\begin{equation}\label{2.25}
\overset{\hh}{\nabla}_k\xi^i=0,\quad \overset{\vv}{\nabla}_k\xi^i=\delta_k^i.
\end{equation}
As can be easily shown, $\overset{\vv}{\nabla}$ and $\overset{\hh}{\nabla}$ are well-defined first-order differential operators. In particular, they extend naturally to the Sobolev space $H^1(\beta^r_s\M)$.

The vertical divergence, $\overset{\vv}{\dive}$, and the horizontal divergence, $\overset{\hh}{\dive}$, of a semibasic vector field $V$ are defined by
\begin{equation}\label{2.26}
 \overset{\vv}{\dive}(V)= \overset{\vv}{\nabla}_k v^k,\quad  \overset{\hh}{\dive}(V)= \overset{\hh}{\nabla}_k v^k.
\end{equation}
\smallskip

To prove Theorem \ref{theorem}, we also need the two following divergence formulas   (see \cite{[Sh]}, p 101).\\
\textit{Gauss-Ostrogradskii formula} of the vertical divergence
\begin{equation}\label{2.27}
\int_{S\M} \overset{\vv}{\dive}(W)\, \dvv=(n-2) \int_{S\M} \langle W,\xi\rangle \, \dvv,\quad W\in \mathcal{C}^\infty(T\M),
\end{equation}
\textit{Gauss-Ostrogradskii formula} of the horizontal divergence
\begin{equation}\label{2.28}
\int_{S\M} \overset{\hh}{\dive}(V)\, \dvv= \int_{\p S\M} \langle V,\nu\rangle \, \dss, \quad V\in \mathcal{C}^\infty(T\M),
\end{equation}
Let $H$ denote the vector field associated with the geodesic flow $\phi_t$. For $u\in \mathcal{C}^\infty(S\M)$ and $(x,\xi)\in S\M$, we have
\begin{equation}\label{2.29}
Hu(x,\xi)=\frac{d}{dt}u(\phi_t(x,\xi))_{|t=0}.
\end{equation}    
and we call it the differentiation along the geodesics. In coordinate form, we have 
\begin{equation}\label{2.30}
H=\xi^i\frac{\p}{\p x^i}-\Gamma^i_{jk}\xi^j\xi^k\frac{\p}{\p \xi^i}=\xi^i (\frac{\p}{\p x^i}-\Gamma^p_{iq}\xi^q \frac{\p}{\p \xi^p})=\xi^i\overset{\hh}{\nabla}_i.
\end{equation}
Now we consider the \textit{Pestov identity}, which is the basic energy identity that has been used since the work of Mukhometov \cite{[Mukh]} in most injectivity proofs of ray transforms in absence of real-analyticity or special symmetries. For a function $u \in \mathcal{C}^{\infty}(T \M)$, we have
\begin{equation}\label{2.31}
2\langle  \overset{\hh}{\nabla} u,  \overset{\vv}{\nabla}Hu \rangle=\vert  \overset{\hh}{\nabla} u \vert^2+  \overset{\hh}{\dive}(V) + \overset{\vv}{\dive}(W)- \langle R(\xi,\overset{\vv}{\nabla} u)\xi,  \overset{\vv}{\nabla} u\rangle.
\end{equation}
Here $H$ is the geodesic vector field associated with the geodesic flow given by (\ref{2.29}), $R$ is the curvature tensor, and the semibasic vector $V$ and $W$ are given by
 \begin{equation}\label{2.32}
  V = \langle \overset{\hh}{\nabla} u, \overset{\vv}{\nabla} u\rangle\xi-\langle\xi, \overset{\hh}{\nabla} u\rangle  \overset{\vv}{\nabla} u,
\end{equation}   
 \begin{equation}\label{2.33}
 W=\langle\xi,\overset{\hh}{\nabla} u\rangle  \overset{\hh}{\nabla} u.
 \end{equation}
 We introduce the function $u:S\M\to\R$ defined by 
   \begin{equation}\label{2.34}
    u(x,\xi)=\int_0^{\ell_+(x,\xi)}f(\gamma_{x,\xi}(t)) dt.
   \end{equation}
  It satisfies the boundary conditions 
  \begin{equation}\label{2.35}
  u=\I f,\quad\mbox{on }\p_+ S\M,
  \end{equation}
    and, since $\ell_+(x,\xi)=0$ for $(x,\xi)\in \p_- S\M$, we have
    \begin{equation}\label{2.36}
     u=0,\quad\mbox{on }\p_- S\M.
\end{equation}
\begin{lemma}\label{L.2.2}
The function $u$ given by (\ref{2.34}) is smooth on $T\M\backslash T(\p\M)$ and has the following properties:
\begin{enumerate}
\item  $u$ is homogeneous function of degree $-1$  in $\xi$.
\item  $u$ satisfies the following kinetic equation $Hu=-f$.
\item $u$ satisfies the following equation $H\overset{\vv}{\nabla} u=-\overset{\hh}{\nabla} u$.
\end{enumerate}
\end{lemma}
\begin{proof}
\begin{itemize}
Item (1) is immediate from the relations $\ell_+(x,\lambda \xi)=\frac{1}{\lambda}\ell_+(x,\xi)$ and $\gamma_{x,\lambda \xi}(t)=\gamma_{x,\xi}(\lambda t)$ for any $\lambda>0$. Then 
$$
u(x,\lambda \xi)=\int_0^{\frac{1}{\lambda}\ell_+(x,\xi)}f(\gamma_{x,\xi}(\lambda t)) dt = \frac{1}{\lambda} u(x,\xi).
 $$
Prove item (2). For $s \in \R$ sufficiently small, we set  $x_s=\gamma_{x,\xi}(s)$ and $\xi_s=\dot{\gamma}_{x,\xi}(s).$ Then, $\gamma_{x_s,\xi_s}(t)=\gamma_{x,\xi}(t+s)$ and $\ell_+(x_s,\xi_s)=\tau_+(x,\xi)-s.$ So,
$$
 u(\gamma_{x,\xi}(s),\dot{\gamma}_{x,\xi}(s))=u(x_s,\xi_s) = \int_0^{\ell_+(x_s,\xi_s)}f(\gamma_{x,\xi}(t+s)) dt =\int_s^{\ell_+(x,\xi)}f(\gamma_{x,\xi}(t)) dt .
$$
 Differentiating with respect to $s$ and taking $s=0$, we obtain that 
 $$
 \frac{\p u}{\p x_i}\dot{\gamma}_{x,\xi}^i(0)+\frac{\p u}{\p \xi_i}  \ddot {\gamma}_{x,\xi}^{i}(0)=-f(x).
 $$
 Since we have $\gamma_{x,\xi}(0)=x$, $\dot{\gamma}_{x,\xi}(0)=\xi$ and $\ddot{\gamma}_{x,\xi}^{i}(0)=-\Gamma ^i_{jk}(x)\xi^j\xi^k,$ then $$\xi^i\frac{\p u}{\p x_i}-\Gamma^i_{jk}\xi^j\xi^k\frac{\p u}{\p \xi_i}= -f(x).$$ Thus we have $Hu=-f$.\\
To prove item (3), applying the operator $\overset{\vv}{\nabla}$ to the kinetic equation, we obtain $ \overset{\vv}{\nabla} (Hu)=-\overset{\vv}{\nabla} f= 0.$ It follows that 
$$
 0= \overset{\vv}{\nabla} (Hu)  = \overset{\vv}{\nabla}_j (\xi^i  \overset{\hh}{\nabla}_i u) dx^j  =(\overset{\vv}{\nabla}_j \xi^i ) \overset{\hh}{\nabla}_i u dx^j+\xi ^i  (\overset{\vv}{\nabla}_j \overset{\hh}{\nabla}_i u) dx^j.
$$
From (\ref{2.24}) and (\ref{2.25}), we get 
$$
0=\overset{\hh}{\nabla} u+H(\overset{\vv}{\nabla}u).
$$
\end{itemize}
The proof is complete.
\end{proof}
To prove the Theorem \ref{theorem}, we will also need the following lemma (see  \cite{[Sh]}, pp. 124, for the proof).
\begin{lemma}\label{L.2.3}
Let $(\M,\g)$ be a simple Riemannian manifold and $K^+$ given by (\ref{2.8}). Let a semibasic tensor field $u\in\mathcal{C}^\infty(\beta_m^0\M)$ satisfies boundary condition (\ref{2.36}), then the following inequality 
\begin{equation}\label{2.37}
\int_{S\M} K^+(x,\xi) \vert u(x,\xi) \vert ^2 \dvv \leq k^+\int_{S\M} \vert Hu(x,\xi) \vert ^2 \dvv.
\end{equation}
holds true. Here $k^+=k^+(\M,\g)$ is given by (\ref{2.9}).
\end{lemma}
\subsection*{Proof of Theorem \ref{theorem}}
We suppose, for a moment, that we have proved the equality
\begin{equation}\label{2.38}
\int_{S\M} \left[ \vert \overset{h}{\nabla} u \vert ^2-\langle R(\xi,\overset{\vv}{\nabla} u)\xi,  \overset{\vv}{\nabla} u\rangle+(n-2) \vert Hu\vert ^2 \right] \dvv= -\int_{\p_+S\M} \langle V,\nu\rangle \dss, 
\end{equation}
and the estimates 
\begin{equation}\label{2.39}
\vert \int_{\p_+S\M}\langle V,\nu\rangle  \dss \vert \leq C \Vert \I f \Vert^2_{H^1(\p_+S\M)},
\end{equation}
and 
\begin{equation}\label{2.40}
\vert \int_{S\M} \langle R(\xi,\overset{\vv}{\nabla} u)\xi,  \overset{\vv}{\nabla} u\rangle  \dvv \vert \leq k^+  \int_{S\M}\vert H \overset{\vv}{\nabla}u\vert^2 \dvv.
\end{equation}
Combining (\ref{2.40}), (\ref{2.39}) and (\ref{2.38}) with item (3) of Lemma \ref{L.2.2}, we find 
\begin{equation}\label{2.41}
(1-k^+)\int_{S\M}  \vert \overset{\hh}{\nabla} u \vert ^2 \dvv+(n-2) \int_{S\M} \vert Hu\vert ^2  \dvv\leq C \Vert \I f \Vert^2_{H^1(\p_+S\M)}.
\end{equation}
For $k^+<1$ and $n\geq 2$, we deduce the estimate
\begin{equation}\label{2.42}
\int_{S\M}  \vert \overset{\hh}{\nabla} u \vert ^2 \dvv \leq  C \Vert \I f \Vert^2_{H^1(\p_+S\M)}.
\end{equation}
In view of the definition of $H$, in (\ref{2.30}), there exists a constant $C$ such that 
$$
\vert Hu \vert^2 \leq C  \vert \overset{\hh}{\nabla} u \vert ^2 .
$$ 
Using the item (2) of Lemma \ref{L.2.2} and (\ref{2.30}), we conclude that 
$$
\Vert f \Vert ^2_{L^2(S\M)}=\int_{S\M} \vert Hu \vert^2 \dvv\leq C \int_{S\M}  \vert \overset{\hh}{\nabla} u \vert ^2 \dvv \leq  C\Vert\I f \Vert^2_{H^1(\p_+S\M)},
 $$ and the Theorem \ref{theorem} is done. \\   
 Now, we come back to prove (\ref{2.38})-(\ref{2.39}) and (\ref{2.40}). First, we start with (\ref{2.40}). By (\ref{2.8}) we find
 \begin{equation}\label{2.43}
 \vert \langle R(\xi,\overset{\vv}{\nabla} u)\xi,  \overset{\vv}{\nabla} u\rangle \vert \leq K^+(x,\xi) \vert \overset{\vv}{\nabla} u(x,\xi) \vert^2,\quad \forall (x,\xi)\in S\M.
 \end{equation}
  Furthermore, the lemma \ref{L.2.3} combined with (\ref{2.43}) gives the following estimate
\begin{align} \label{2.44}
\int_{S \M} \vert \langle R(\xi,\overset{\vv}{\nabla} u)\xi,  \overset{\vv}{\nabla} u\rangle \vert \dvv & \leq k^+ \int_{S \M}  \vert H\overset{\vv}{\nabla} u \vert^2 \dvv & \cr & \leq k^+ \int_{S \M}  \vert \overset{\hh}{\nabla} u \vert^2 \dvv,
\end{align} 
This completes the proof of (\ref{2.40}).\\
We prove now (\ref{2.38}).  Since we have $Hu=-f$ and  $\overset{\vv}{\nabla}f=0$, then the Pestov's identity (\ref{2.31}) gives 
 \begin{equation}\label{2.45}
 \vert  \overset{\hh}{\nabla} u \vert^2+  \overset{\hh}{\dive}(V) + \overset{\vv}{\dive}(W)- \langle R(\xi,\overset{\vv}{\nabla} u)\xi,  \overset{\vv}{\nabla} u\rangle=0.
 \end{equation}
 Avoiding eventual singularities of $u$  on $T \p\M$, we will consider the variety $\M_{\rho}$ defined by $$\M_{\rho}=\{x \in \M,\quad d_\g(x,\p \M) \geq \rho \},$$ where $\rho >0$. In some neighbourhood of $\p \M,$ the function $x \mapsto d_\g(x,\p \M)$ is smooth and $\p \M_{\rho}$ is strictly convex for sufficiently small $\rho >0.$ The function $u$ is smooth on $S\M_{\rho}$ since $S\M_{\rho}\subset S\M \setminus S(\p \M)$. Integrating  (\ref{2.45}) over $S\M_{\rho}$ and using the formula divergence (\ref{2.27}) and (\ref{2.28}), we find 
 \begin{align*}
 \int_{S\M_{\rho}}\left[ \vert  \overset{\hh}{\nabla} u \vert^2- \langle R(\xi,\overset{\vv}{\nabla} u)\xi,  \overset{\vv}{\nabla} u\rangle\right]\, \dvv= &
 -\int_{S\M_{\rho}} \left[ \overset{\hh}{\dive}(V)  + \overset{\vv}{\dive}(W)\right]\, \dvv \cr =&
  -\int_{\p S\M_{\rho}} \langle V, \nu \rangle \quad d\sigma^{2n-2} -(n-2)\int_{S\M_{\rho}} \langle W, \xi \rangle \dvv,
 \end{align*}
 where $\nu=\nu_{\rho}(x)$ is the unit vector of the outer normal to the boundary of $\M_{\rho}$.
 In view of (\ref{2.33}), we have 
 $$\langle W, \xi \rangle=\langle\xi,\overset{\hh}{\nabla} u\rangle ^2 = \vert Hu\vert ^2.$$
Hence, we obtain the equality 
\begin{equation}\label{2.46}
 \int_{S\M_{\rho}}\left[ \vert  \overset{\hh}{\nabla} u \vert^2- \langle R(\xi,\overset{\vv}{\nabla} u)\xi,  \overset{\vv}{\nabla} u\rangle+(n-2) \vert Hu\vert^2 \right]  \dvv= -\int_{\p S\M_{\rho}} \langle V, \nu \rangle \,d\sigma^{2n-2}.
\end{equation} 
Now, we wish pass to the limit as $\rho \rightarrow 0.$ We will apply the Lebesgue dominated convergence theorem.  Denote by $\mathbb{\chi}_\rho$ the characteristic function of the set $S\M_{\rho}$ and by $p$  the projection $p : \p S\M \longrightarrow  \p S\M_\rho$, $p(x,\xi)=(x',\xi'),$ where $x'$ 
 is such that the geodesic $\gamma_{xx'}$ has length $\rho$ and intersects $\p \M$ orthogonally at $x$ and $x'$, and $\xi'$ is obtained 
 by the parallel translation of the vector $\xi$ along $\gamma_{xx'}$. So the equality (\ref{2.46}) becomes 
 \begin{equation}\label{2.47}
 \int_{S\M}\left[ \vert  \overset{\hh}{\nabla} u \vert^2- \langle R(\xi,\overset{\vv}{\nabla} u)\xi,  \overset{\vv}{\nabla} u\rangle+(n-2) \vert Hu\vert ^2 \right] \mathbb{\chi}_\rho  \dvv= -\int_{\p S\M} \langle V, \nu \rangle p_* ( d\sigma^{2n-2}).
\end{equation} 
Note that all the integrands of (\ref{2.47}) are smooth on $S\M \setminus \p S \M$ and so, they converge towards their values almost everywhere, when $\rho \rightarrow 0$. 
Since the functions $\vert  \overset{\hh}{\nabla} u \vert^2$ and $\vert Hu\vert ^2$  are positive and the second function satifies (\ref{2.43}), then the left side of (\ref{2.47}) converges as $\rho \rightarrow 0$.  To apply the Lebesgue dominated convergence theorem in (\ref{2.47}), it remains to prove that $\vert \langle V, \nu \rangle p_*  \vert$ is bounded by a summable function on $\p S \M$ which does not depend on $\rho$.\\
For $(x,\xi)\in \p S\M$, we put
$$
\overset{\hh}{\nabla}_{\textrm{tan}} u =\overset{\hh}{\nabla} u -\langle\overset{\hh}{\nabla} u ,\nu\rangle\nu,\quad \overset{\vv}{\nabla}_{\textrm{tan}} u =\overset{\vv}{\nabla} u -\langle\overset{\vv}{\nabla} u,\xi\rangle\xi.
$$  
We see that
$$
\langle\overset{\hh}{\nabla}_{\textrm{tan}} u ,\nu\rangle=\langle\overset{\vv}{\nabla}_{\textrm{tan}} u,\xi\rangle=0.
$$
Then $\overset{\hh}{\nabla}_{\textrm{tan}}$ and $\overset{\vv}{\nabla}_{\textrm{tan}}$ are in fact  differential operators on $\p S\M$ and $\overset{\hh}{\nabla}_{\textrm{tan}}u$, $\overset{\vv}{\nabla}_{\textrm{tan}}u$ are completely determined by the restriction $u_{|\p S\M}$ of $u$ on $\p S\M$.\\
For $(x,\xi)\in\p S\M$, by a simple computation we obtain
\begin{equation}\label{2.48}
\langle V, \nu \rangle=\langle \overset{\hh}{\nabla}_{\textrm{tan}} u , \overset{\vv}{\nabla}_{\textrm{tan}} u  \rangle\langle\xi,\nu\rangle-\langle \overset{\hh}{\nabla}_{\textrm{tan}} u , \xi\rangle\langle \overset{\vv}{\nabla}_{\textrm{tan}}  u ,\nu\rangle.
\end{equation}
From (\ref{2.34}), we can see that the derivatives $\overset{\hh}{\nabla}_{\textrm{tan}}u$ and $\overset{\vv}{\nabla}_{\textrm{tan}} u$  are locally bounded on $\p S\M$. It is important that the right-hand side of (\ref{2.48}) does not contain $\langle\overset{\hh}{\nabla} u,\nu\rangle$ and $\langle\overset{\vv}{\nabla}u,\xi \rangle.$\\
Taking $\rho \rightarrow 0$ in the equality (\ref{2.46}), we have 
$$
\int_{S\M} \left[ \vert \overset{\hh}{\nabla} u \vert ^2-\langle R(\xi,\overset{\vv}{\nabla} u)\xi,  \overset{\vv}{\nabla} u\rangle+
(n-2) \vert Hu\vert ^2 \right]\dvv= -\int_{\p S\M} \langle V, \nu \rangle \dss.
$$
Finally, It remains to prove the estimate (\ref{2.39}).  In view of the boundary condition $u=\I f$ on $\p_+S\M$ and $u=0$ on $\p_-S\M$ we obtain
\begin{align*}
 \int_{\p S\M} \langle V, \nu \rangle \dss &=\int_{\p_+ S\M} \para{\langle \overset{\hh}{\nabla}_{\textrm{tan}} (\I f) , \overset{\vv}{\nabla}_{\textrm{tan}} (\I f)  \rangle\langle\xi,\nu\rangle-\langle \overset{\hh}{\nabla}_{\textrm{tan}} (\I f) , \xi\rangle\langle \overset{\vv}{\nabla}_{\textrm{tan}} (\I f) ,\nu\rangle} \dss\cr
& :=\int_{\p_+ S\M} \mathcal{Q}(\I f) \dss .
 \end{align*}
where $\mathcal{Q}u$ is a quadratic form in variables $\overset{\hh}{\nabla}_{\textrm{tan}}u$ and $\overset{\vv}{\nabla}_{\textrm{tan}}u$ and hence, $\mathcal{Q}$ is a quadratic first-order  differential operator on the manifold $\p_+ S \M$. Consequently, there exists a constant $C$ such that we have 
$$
\vert  \int_{\p_+ S \M} \mathcal{Q}(\I f) \dss \vert \leq C \Vert \I f \Vert ^2_{H^1(\p_+ S \M)}.
$$
 This completes the proof of the Theorem \ref{theorem}.
\section{Geometric optics solutions for the damped wave equation}
\setcounter{equation}{0}
The main result in this section is Lemma \ref{L.3.2}, which ensures the existence of a familly of solutions of the wave equation.
\smallskip

The WKB expansion method is a classical way to construct a special solution with a large parameter 
of wave systems. It is based on the assumption that the solution of the wave equation, can be sought as an expansion in powers of the frequency.
This expansion arises here as a power series depending on the small parameter $h$, which represents the
relative wavelength of the initial conditions. Introducing it in a scalar  wave equation leads to a system of coupled equations governing the behavior of the phase (eikonal equation) and of the amplitudes of the different expansion coefficients (transport equations).
\smallskip

 Denote by $\dive X$ the divergence of a vector field $X\in H^1(\M,T\M)$ on $\M$, i.e. in local coordinates (see pp. 42, \cite{[KKL]}),
\begin{equation}\label{3.1}
\dive X=\frac{1}{\sqrt{\abs{\g}}}\sum_{i=1}^n\p_i\para{\sqrt{\abs{\g}}\,X^i},\quad X=\sum_{i=1}^nX^i\p_i, \quad \abs{\g}=\det\g.
\end{equation}
If $X\in H^1(\M,T\M)$ the divergence formula reads
\begin{equation}\label{3.2}
\int_\M\dive X \dv=\int_{\p \M}\seq{X,\nu} \ds,
\end{equation}
and for a function $f\in H^1(\M)$ Green's formula reads
\begin{equation}\label{3.3}
\int_\M\dive X\,f\dv=-\int_\M\seq{X,\nabla f} \dv+\int_{\p \M}\seq{X,\nu} f\ds.
\end{equation}
Then if $f\in H^1(\M)$ and $w\in H^2(\M)$, the following identity holds
\begin{equation}\label{3.4}
\int_\M\Delta w f\dv=-\int_\M\seq{\nabla w,\nabla f} \dv+\int_{\p \M}\p_\nu w f \ds.
\end{equation}
In this section we give a construction of geometric optics solutions to the wave equation which concentrated along geodesic curves in space-time. We are, however, dealing with damped equations.
For $T>\textrm{Diam}_\g(\M)$, let $\M_1$ a simple Riemannian manifold and, $\varepsilon>0$ such that
\begin{equation}\label{3.5}
\M_1\supset\overline{\M},\quad T>\textrm{Diam}_\g(\M_1)+2\varepsilon.
\end{equation}
The absorption coefficients $a_1$, $a_2$ and the potentials $q_1$ and $q_2$ may be extended to $\M_1$.\\
Let $y\in \p \M_1$. Denote points in $\M_1$ by $(r,\xi)$
where $(r,\xi)$ are polar normal coordinates in $\M_1$ with center
$y$. That is
\begin{equation}\label{3.6}
x=\exp_{y}(r\xi),\quad r=d_\g(y,x)>0,\quad \xi\in S_{y}\M_1=\set{\xi\in T_{y}\M_1,\,\,\abs{\xi}=1}.
\end{equation}
In these coordinates (which depend on the choice of $y$) the metric takes the form
$$
\widetilde{\g}(r,\xi)=\dd r^2+\g_0(r,\xi),
$$
where $\g_0(r,\xi)$ is a smooth positive definite metric.
For any function $w$ compactly supported in $\M$, we set for $r>0$ and $\xi\in S_y\M_1$
$$
\widetilde{w}(r,\xi)=w(\exp_{y}(r\xi)),
$$
where we have extended $w$ by $0$ outside $\M$.\\
Finally we denote 
$$
S_y^+\M_1=\set{\xi\in S_y\M_1,\,\, \seq{\nu(y),\xi}<0}.
$$
\subsection{Eikonal and transport equations}
We start with the following lemma which give a solution of an eikonal equation
\begin{lemma}\label{L.3.1}
 For $y\in\p\M_1$, we denote the function $\varrho(x)=d_\g(y,x)$. Then $\varrho\in\mathcal{C}^2(\M)$ and satisfies the eikonal equation 
\begin{equation}\label{3.7}
\abs{\nabla\varrho}^2=\g^{ij}\frac{\p\varrho}{\p x_i}\frac{\p\varrho}{\p
x_j}=1,\qquad \forall x\in \M.
\end{equation}
\end{lemma}
\begin{proof}
By the simplicity assumption, since $y\notin\overline{\M}$, we have $\varrho\in\mathcal{C}^\infty(\M)$, and in polar normal coordinates, we get
\begin{equation}\label{3.8}
\widetilde{\varrho}(r,\xi)=r=d_\g(y,x).
\end{equation}
The proof is complete.
\end{proof}
Let us introduce the following spaces:
\begin{equation}\label{3.9}
\mathcal{V}(Q)=\set{\theta\in H^3(0,T;L^2(\M))\cap H^1(0,T;H^2(\M)),\,\,\p_t^j\theta(\cdot,0)=\p_t^j\theta(\cdot,T)=0,\,j=0,1,2.},
\end{equation}
and
\begin{equation}\label{3.10}
\mathcal{W}(Q)=\set{\psi\in W^{2,\infty}(Q),\,\, \p_t\psi\in W^{2,\infty}(Q)},
\end{equation}
equipped with the norms:
\begin{eqnarray*}
\norm{\theta}_{\mathcal{V}(Q)}&:=&\norm{\theta}_{H^1(0,T;H^2(\M))}+\norm{\theta}_{H^3(0,T;L^2(\M))}, \quad \theta\in\mathcal{V}(Q),\cr
\\
\norm{\psi}_{\mathcal{W}(Q)}&:=&\norm{\psi}_{W^{2,\infty}(Q)}+\norm{\p_t\psi}_{W^{2,\infty}(Q)},\quad \psi\in\mathcal{W}(Q).
\end{eqnarray*}
We want to find a function $\theta\in \mathcal{V}(Q)$ which solves the first transport equation
\begin{equation}\label{3.11}
\p_t \theta+\seq{d\varrho,d\theta}+\frac{1}{2} (\Delta \varrho)\theta=0,\qquad \forall t\in\R,\, x\in\M,
\end{equation}
where $\varrho$ is given by Lemma \ref{L.3.1}.\\
Moreover for $a\in W^{2,\infty}(\M) $, we need to find a function  
$\psi_a\in\mathcal{W}(Q) $ which solves the second transport equation
\begin{equation}\label{3.12}
\p_t \psi_a+\seq{d\varrho,d\psi_a}+\frac{a}{2}\psi_a=0,\qquad \forall t\in\R,\, x\in\M.
\end{equation}

The first step is to solve the transport equation (\ref{3.11}). Recall that
if $f(r)$ is any function of the geodesic distance $r$, then
\begin{equation}\label{3.13}
\Delta_{\widetilde{\g}}f(r)=f''(r)+\frac{\alpha^{-1}}{2}\frac{\p
\alpha}{\p r}f'(r).
\end{equation}
Here $\alpha=\alpha(r,\xi)$ denotes the square of the volume element in geodesic polar coordinates.
The transport equation (\ref{3.11}) becomes
\begin{equation}\label{3.14}
\frac{\p \widetilde{\theta}}{\p t}+\frac{\p \widetilde{\varrho}}{\p
r}\frac{\p \widetilde{\theta}}{\p
r}+\frac{1}{4}\widetilde{\theta}\alpha^{-1}\frac{\p \alpha}{\p r}\frac{\p
\widetilde{\varrho}}{\p r}=0.
\end{equation}
Thus $\widetilde{\theta}$ satisfies
\begin{equation}\label{3.15}
\frac{\p \widetilde{\theta}}{\p t}+\frac{\p \widetilde{\theta}}{\p
r}+\frac{1}{4}\widetilde{\theta}\alpha^{-1}\frac{\p \alpha}{\p r}=0.
\end{equation}
Let $\phi\in\mathcal{C}_0^\infty(\R)$ and $\Psi\in H^2(S^+_y\M)$. Let us write $\widetilde{\theta}$ in the form
\begin{equation}\label{3.16}
\widetilde{\theta}(t,r,\xi)=\alpha^{-1/4}\phi(t-r)\Psi(\xi).
\end{equation}
Direct computations yield
\begin{equation}\label{3.17}
\frac{\p \widetilde{\theta}}{\p
t}(t,r,\xi)=\alpha^{-1/4}\phi'(t-r)\Psi(\xi),
\end{equation}
and, we find
\begin{equation}\label{3.18}
\frac{\p \widetilde{\theta}}{\p
r}(t,r,\xi)=-\frac{1}{4}\alpha^{-5/4}\frac{\p\alpha}{\p
r}\phi(t-r)\Psi(\xi)-\alpha^{-1/4}\phi'(t-r)\Psi(\xi).
\end{equation}
Finally, (\ref{3.18}) and (\ref{3.17}) yield
\begin{equation}\label{3.19}
\frac{\p \widetilde{\theta}}{\p t}(t,r,\xi)+\frac{\p \widetilde{\theta}}{\p
r}(t,r,\xi)=-\frac{1}{4}\alpha^{-1}\widetilde{\theta}(t,r,\xi)\frac{\p\alpha}{\p
r}.
\end{equation}
Now if we assume that $\mathrm{supp}(\phi) \subset (0,\epsilon)$,  then for any $x=\exp_y(r\xi)\in \M$, it is easy to
see that 
$$
\p_t^j\widetilde{\theta}(0,r,\xi)=\p_t^j\widetilde{\theta}(T,r,\xi)=0,\quad j=0,1,2,\,\, T-r>\varepsilon.
$$
For the second transport equation (\ref{3.12}), in polar coordinates, takes the form
\begin{equation}\label{3.20}
\frac{\p \widetilde{\psi}_a}{\p t}+\frac{\p \widetilde{\varrho}}{\p
r}\frac{\p \widetilde{\psi}_a}{\p
r} +\frac{1}{2}\widetilde{a}(r,y,\xi)\widetilde{\psi}_a=0,
\end{equation}
where $\widetilde{a}(r,y,\xi):=a(\theta_r(y,\xi))$. Thus $\widetilde{\psi}_a$ satisfies
\begin{equation}\label{3.21}
\frac{\p \widetilde{\psi}_a}{\p t}+\frac{\p \widetilde{\psi}_a}{\p
r}+\frac{1}{2}\widetilde{a}(r,y,\xi)\widetilde{\psi}_a=0.
\end{equation}
Thus, we can choose $\widetilde{\psi}_a$ as following
\begin{equation}\label{3.22}
\widetilde{\psi}_a(t,y,r,\xi)=\exp\para{-\frac{1}{2}\int_0^t\widetilde{a}(r-s,y,\xi)ds}.
\end{equation}
Since $a\in W^{2,\infty}(\M)$, we get $\psi_a\in \mathcal{W}(Q)$. Hence (\ref{3.12}) is solved.
\subsection{WKB-solutions of the wave equation}
We  introduce the function
\begin{equation}\label{3.23}
\varphi(x,t):=\varrho(x)-t,\quad x\in\M,\,\, t\in (0,T),
\end{equation}
where $\varrho$ is given by Lemma \ref{L.3.1}.
\begin{lemma}\label{L.3.2}
 Let  $a\in W^{2,\infty}(\M)$, $q\in W^{2,\infty}(\M)$, and $\theta\in\mathcal{V}(Q)$, $\psi_{a}\in\mathcal{W}(Q)$ solve respectively (\ref{3.11}) and (\ref{3.12}). Then for all $h>0$ small enough, there exists a solution 
 $$
 u(x,t;h)\in \mathcal{C}^2(0,T;L^2(\M))\cap\mathcal{C}^1(0,T;H^1(\M))\cap\mathcal{C}(0,T;H^2(\M))
 $$ 
 of the wave equation
\begin{equation*}
(\p^2_t-\Op)u=0,\quad \textrm{in}\quad Q,
\end{equation*}
with the initial condition
\begin{equation*}
u(x,0)=\p_tu(x,0)=0,\quad \textrm{in}\quad \M,
\end{equation*}
 of the form
\begin{equation}\label{3.24}
u(x,t)=\theta(x,t)\psi_a(x,t)e^{i\varphi(x,t)/h}+r_h(x,t),
\end{equation}
the remainder $r_h(x,t)$ is such that
\begin{align*}
r_h(x,t)&=0,\quad (x,t)\in \Sigma , \\
r_h(x,0)=\p_t r_h(x,0)&=0,\quad x\in \M.
\end{align*}
Furthermore, there exist $C>0$, $h_0>0$ such that, for all $h \leq h_0$ the following estimates hold true.
\begin{equation}\label{3.25}
\sum_{k=0}^2 \sum_{j=0}^{k} h^{k-1} \|\p_t^{j} r_h(\cdot,t)\|_{H^{k-j}(\M)}\leq C\norm{\theta}_{\mathcal{V}(Q)}.
\end{equation}
The constant $C$ depends only on $T$ and $\M$ (that is $C$ does
not depend on $a$ and $h$). 
\end{lemma}
\begin{proof}
Let $r(x,t;h)$ solves the following 
homogenous boundary value problem
\begin{equation}\label{3.26}
\left\{
\begin{array}{llll}
\para{\partial^2_t-\Op}r(x,t)=V_h (x,t)
& \textrm{in }\,\, Q,\cr
r(x,0)=\p_tr(x,0)=0,& \textrm{in
}\,\,\M,\cr
r(x,t)=0 & \textrm{on} \,\, \Sigma.
\end{array}
\right.
\end{equation}
where the source term $V_h$ is given by 
\begin{equation}\label{3.27}
V_h(x,t)=-\para{\partial_t^2-\Op}\para{(\theta\psi_a)(x,t)e^{i\varphi/h}}.
\end{equation}
To prove our Lemma it would be enough to show that $r$ satisfies the estimates (\ref{3.25}).\\
By a simple computation, we have
\begin{align}\label{3.28}
-V_h (x,t)&=e^{i\varphi(x,t)/h}\para{\p_t^2-\Op}\para{(\theta\psi_a)( x,t)}\cr
&\quad -\frac{2i}{h} e^{i\varphi(x,t)/h}\psi_a(x,t)\para{\p_t\theta+\seq{d\varrho,d\theta}+\frac{\theta}{2}\Delta\varrho}( x,t)\cr
&\quad -\frac{2i}{h} e^{i\varphi(x,t)/h}\theta(x,t)\para{\p_t\psi_a+\seq{d\varrho,d\psi_a}+\frac{a}{2} \psi_a}(x,t)\cr
&\quad-\frac{1}{h^2} \theta\psi_a(x,t) e^{i\varphi(x,t)/h}\para{1-\abs{d\varrho}^2}.
\end{align}
Taking into account  (\ref{3.7})-(\ref{3.11}) and (\ref{3.12}), the right-hand side of (\ref{3.28}) becomes
\begin{align}\label{3.29}
V_h (x,t)&=-e^{i\varphi(x,t)/h}(\p_t^2-\Op)\para{(\theta\psi_a)(x,t)}
\cr &\equiv-e^{i\varphi(x,t)/h}V_0(x,t).
\end{align}
Since $\theta\in\mathcal{V}(Q)$ and $\psi_a\in\mathcal{W}(Q)$ we deduce that $V_0\in H^1_0(0,T;L^2(\M))$. 
Furthermore, there is a constant $C>0$, such that
\begin{equation}\label{3.30}
\norm{V_0}_{L^2(Q)}+\norm{\p_tV_0}_{L^2(Q)}\leq C\norm{\theta}_{\mathcal{V}(Q)}.
\end{equation}
By Lemma \ref{L.1.1}, we find
\begin{equation}\label{3.31}
r_h\in \mathcal{C}^2(0,T;L^2(\M))\cap\mathcal{C}^1(0,T;H^1_0(\M))\cap\mathcal{C}(0,T;H^2(\M)).
\end{equation}
Since the coefficients $a$ and $q$ do not depend on $t$, the function
$$
r^*_h(x,t)=\int_0^tr_h(x,s)ds,
$$
solves the mixed hyperbolic problem (\ref{3.26}) with the right side
$$
V^*_h(x,t)=\int_0^tV_h (x,s)ds=-ih\int_0^tV_0(x,s)\p_s\para{e^{i\varphi(x,s)/h}}ds.
$$
Integrating by part with respect to $s$, we conclude that
$$
\norm{V^*_h}_{L^2(Q)}\leq Ch\norm{\theta}_{\mathcal{V}(Q)}.
$$
and by (\ref{1.5}), we get
\begin{eqnarray}\label{3.32}
\norm{r_h(\cdot,t)}_{L^2(\M)}=\norm{\p_tr^*_h(\cdot,t)}_{L^2(\M)}\leq
Ch\norm{\theta}_{\mathcal{V}(Q)}.
\end{eqnarray}
Since $\norm{V_h }_{L^2(Q)}+h\norm{\p_tV_h }_{L^2(Q)}\leq C\norm{\theta}_{\mathcal{V}(Q)}$, by using again the energy estimates for the problem (\ref{3.26}), obtain
\begin{eqnarray}\label{3.33}
\norm{\p_tr_h(\cdot,t)}_{L^2(\M)}+\norm{\nabla r_h (\cdot,t)}_{L^2(\M)}\leq C\norm{\theta}_{\mathcal{V}(Q)}.
\end{eqnarray}
and by (\ref{1.6}), we have
\begin{eqnarray}\label{3.34}
\norm{\p_t^2 r_h(\cdot,t)}_{L^2(\M)}+\norm{\nabla\p_t r_h(\cdot,t)}_{L^2(\M)}+\norm{\Delta r_h(\cdot,t)}_{L^2(\M)}\leq Ch^{-1}\norm{\theta}_{\mathcal{V}(Q)}.
\end{eqnarray}
Collecting (\ref{3.32})-(\ref{3.33}) and (\ref{3.34}) we get (\ref{3.25}).
The proof is complete.
\end{proof}
By similar way, we can prove the following Lemma:
\begin{lemma}\label{L.3.3} Let  $a\in W^{2,\infty}(\M)$, $q\in W^{2,\infty}(\M)$, and $\theta\in\mathcal{V}(Q)$, $\psi_{-a}\mathcal{W}(Q)$ solve respectively (\ref{3.11}) and (\ref{3.12}) (with $a$ replaced by $-a$). Then for all $h>0$ small enough, there exists a solution 
$$
u(x,t;h)\in \mathcal{C}^2(0,T;L^2(\M))\cap\mathcal{C}^1(0,T;H^1(\M))\cap\mathcal{C}(0,T;H^2(\M))
$$ of the wave equation
\begin{equation*}
(\p^2_t-\Opp)u=0,\quad \textrm{in}\quad Q,
\end{equation*}
with the final condition
\begin{equation*}
u(T,x)=\p_tu(T,x)=0,\quad \textrm{in}\quad \M,
\end{equation*}
 of the form
\begin{equation}\label{3.35}
u(x,t;h)=\theta(x,t)\psi_{-a}(x,t)e^{i\varphi(x,t)/h}+r_h(x,t),
\end{equation}
the remainder $r_h(t,x)$ is such that
\begin{align*}
r_h(x,t)&=0,\quad (x,t)\in \Sigma , \\
r_h(x,T)=\p_tr_h(x,T)&=0,\quad x\in \M.
\end{align*}
Furthermore, there exist $C>0$, $h_0>0$ such that, for all $h \leq h_0$ the following estimates hold true.
\begin{equation}\label{3.36}
\sum_{k=0}^2 \sum_{j=0}^{k} h^{k-1} \|\p_t^{j} r_h(\cdot,t)\|_{H^{k-j}(\M)}\leq C\norm{\theta}_{\mathcal{V}(Q)}.
\end{equation}
The constant $C$ depends only on $T$ and $\M$ (that is $C$ does
not depend on $a$ and $h$). 
\end{lemma}
\section{Stable determination of the absorption coefficient}
\setcounter{equation}{0}
In this section, we prove the stability estimate of the absorption coefficient $a$. We are going to use the geometrical
optics solutions constructed in the previous section; this will provide information on the geodesic ray transform of the difference of two absorption coefficients.
\subsection{Preliminary estimates}
The main purpose of this section is to present a preliminary estimate, which relates the
difference of two absorption coefficients to the Dirichlet-to-Neumann map.
As before, we let $a_1,\,a_2\in\mathscr{A}(m_1,\eta)$ and $q_1,q_2\in\mathscr{Q}(m_2)$ such that $a_1=a_2$, $q_1=q_2$ near  the boundary $\p\M$. We set
$$
a(x)=(a_1-a_2)(x),\quad q(x)=(q_1-q_2)(x).
$$
Recall that we have extended $a_{1},a_{2}$ as $W^{2,\infty}(\M_1)$ in such a way that $a=0$ and $q=0$ on $\M_{1} \setminus \M$.\\
We denote $\psi_{a_2}\in\mathcal{W}(Q)$ and $\psi_{-a_1}\in\mathcal{W}(Q)$ the solutions of (\ref{3.12}) respectively with $a=a_2$ and $a=-a_1$ given by (\ref{3.22}), and set
\begin{equation}\label{4.1}
\psi_a(x,t)=\psi_{a_2}(x,t)\psi_{-a_1}(x,t).
\end{equation} 
\begin{lemma}\label{L.4.1}
Let $T>0$. There exist $C>0$ such that for any $\theta_j\in \mathcal{V}(Q)$, $j=1,2$, 
satisfying the transport equation (\ref{3.7}), the following estimate holds true:
\begin{equation}\label{4.2}
\abs{\int_{0}^T\!\!\!\!\int_{\M}a(x)(\theta_2\overline{\theta}_1)(x,t)\psi_{a}(x,t)\,\dv \, \dd t } 
\leq C\para{h+h^{-2}\norm{\Lambda_{a_1,q_1}-\Lambda_{a_2,q_2}}}\norm{\theta_1}_{\mathcal{V}(Q)}\norm{\theta_2}_{\mathcal{V}(Q)}
\end{equation}
for all $h\in (0,h_0)$.
\end{lemma}
\begin{proof} First, if $\theta_2$ satisfies (\ref{3.11}), $\psi_{a_2}$ satisfies (\ref{3.12}), and $h<h_0$, Lemma \ref{L.3.2} guarantees
the existence of a geometrical optics solution $u_2$
\begin{equation}\label{4.3}
u_2(x,t)=(\theta_2\psi_{a_2})(x,t)e^{i\varphi(x,t)/h}+r_{2,h}(x,t),
\end{equation}
to the wave equation corresponding to the coefficients $a_2$ and $q_2$,
$$
\para{\p^2_t-\Ops}u(x,t)=0\quad \textrm{in}\,Q, \quad u(\cdot,0)=\p_tu(\cdot,0)=0 \quad \textrm{in}\,\, \M,
$$
where $r_{2,h}$ satisfies
\begin{gather}\label{4.4}
h^{-1}\norm{r_{2,h}(\cdot,t)}_{L^2(\M)}+\norm{\p_t r_{2,h}(\cdot,t)}_{L^2(\M)}+\norm{\nabla r_{2,h}(\cdot,t)}_{L^2(\M)}\leq C\norm{\theta_2}_{\mathcal{V}(Q)},
\\ \nonumber
r_{2,h}(x,t)=0,\quad\forall (x,t)\in\,\Sigma.
\end{gather}
Moreover
$$
u_2\in  \mathcal{C}^2(0,T;L^2(\M))\cap \mathcal{C}^1(0,T;H^1(\M)) \cap \mathcal{C}(0,T;H^2(\M)).
$$
Let us denote by $f_h$ the function
$$
f_h(x,t)=(\theta_2\psi_{a_2})(x,t)e^{i\varphi(x,t)/h},\quad  (x,t)\in\Sigma,
$$
and we consider $v$ the solution of the following non-homogenous boundary value problem
\begin{equation}\label{4.5}
\left\{\begin{array}{lll}
\para{\p_t^2-\Opf} v=0, & \textrm{in}\,\, Q,\cr
v(x,0)=\p_tv(x,0)=0, & \textrm{in}\,\, \M,\cr
v(x,t)=u_2(x,t):=f_{h}(x,t), & \textrm{on}\,\, \Sigma.
\end{array}
\right.
\end{equation}
We let $ w=v-u_2$. Therefore, $w$ solves the following homogenous boundary value problem
$$
\left\{\begin{array}{lll}
\para{\p_t^2-\Opf}w(x,t)=a(x)\p_tu_2(x,t)+q(x)u_2(x,t) & \textrm{in}\,\,  Q,\cr
w(x,0)=\p_tw(x,0)=0, & \textrm{in}\,\, \M,\cr
w(x,t)=0, & \textrm{on}\,\,  \Sigma.
\end{array}
\right.
$$
Using the fact that  $a(x)\p_tu_2+q(x)u_2 \in W^{1,1}(0,T;L^2(\M))$ with $u_2(\cdot,0)=\p_t u_2(\cdot,0)\equiv 0$, by Lemma \ref{L.1.1}, we deduce that
$$
w\in \mathcal{C}^1(0,T;L^2(\M))\cap \mathcal{C}(0,T;H^2(\M)\cap
H^1_0(\M)).
$$
Therefore, we have constructed a special solution
  $$ 
  \mathcal{C}^2(0,T;L^2(\M))\cap \mathcal{C}^1(0,T;H^1(\M)) \cap \mathcal{C}(0,T;H^2(\M)),
  $$
to the backward wave equation
\begin{align*}
\para{\partial_t^2-\Delta-a_1(x)\p_t+q_1(x)}u_1(x,t)&=0,  \quad (x,t) \in Q, \\
u_1(x,T)=u_1(x,T)&=0,  \quad x \in \M,
\end{align*}
having the special form
\begin{equation}\label{4.6}
u_1(x,t)=(\theta_1\psi_{-a_1})( t,x)e^{i\varphi(x,t)/h}+r_{1,h}(x,t),
\end{equation}
which corresponds to the coefficients $-a_1$ and $q_1$, where
$r_{1,h}$ satisfies for $h<h_0$
\begin{equation}\label{4.7}
h^{-1}\norm{r_{1,h}(\cdot,t)}_{L^2(\M)}+\norm{\p_t r_{1,h}(\cdot,t)}_{L^2(\M)}+\norm{\nabla
r_{1,h}(\cdot,t)}_{L^2(\M)}\leq C\norm{\theta_1}_{\mathcal{V}(Q)}.
\end{equation}
Integrating by parts and using Green's formula (\ref{3.4}), we find
\begin{multline}\label{4.8}
\int_0^T\!\!\!\int_\M\para{\p_t^2-\Opf}w\overline{u}_1\dv \, \dd t
= \int_0^T\!\!\!\int_\M a(x)\p_tu_2\overline{u}_1\dv \, \dd t\cr
+\int_0^T\!\!\!\int_\M q(x)u_2\overline{u}_1\dv \, \dd t
=-\int_0^T\!\!\!\int_{\p \M} \p_\nu w\overline{u}_1\ds \, \dd t.
\end{multline}
Taking (\ref{4.8}), (\ref{4.6}) into account, we deduce
\begin{multline}\label{4.9}
-\int_0^T\!\!\!\int_\M a(x)\p_tu_2\overline{u}_1(x,t)\dv\,\dd t 
=\int_0^T\!\!\!\int_{\p \M}\para{\Lambda_{a_1,q_1}-\Lambda_{q_2,q_2}} f_{h}(x,t)\overline{g}_h(x,t) \ds \, \dd t\cr
+\int_0^T\!\!\!\int_\M q(x)u_2\overline{u}_1\dv \, \dd t
\end{multline}
where $g_h$ is given by
$$
g_h(x,t)=(\theta_1\psi_{-a_1})(x,t)e^{i\varphi(x,t)/h},\quad (x,t)\in \Sigma.
$$
It follows from (\ref{4.9}), (\ref{4.6}) and (\ref{4.3}) that
\begin{multline}\label{4.10}
ih^{-1}\int_0^T\!\!\!\int_\M a(x) (\theta_2 \overline{\theta}_1)(x,t)(\psi_{a_2}\psi_{-a_1})( x,t)\dv\,\dd t =\cr
\int_0^T\!\!\!\int_{\p \M} \overline{g}_h\para{\Lambda_{a_1,q_1}-\Lambda_{a_2,q_2}}f_{h} \ds \, \dd t
-ih^{-1}\int_0^T\!\!\!\int_\M a(x) (\theta_2\psi_{a_2})(x,t)\overline{r}_{1,h}e^{i\varphi/h}\dv\dd t\cr
+\int_0^T\!\!\!\int_\M a(x)\p_t(\theta_2\psi_{a_2})(x,t)\overline{\theta}_1\psi_{a_1}(x,t)\dv\dd t
+\int_0^T\!\!\!\int_\M a(x)\p_t(\theta_2\psi_{a_2})(x,t)\overline{r}_{1,h}(t,x)e^{i\varphi/h}\dv\dd t\cr
+\int_0^T\!\!\!\int_\M a(x)\p_t r_{2,h}(\overline{\theta}_1\psi_{a_1})(x,t) e^{-i\varphi/h}\dv\dd t
+\int_0^T\!\!\!\int_\M a(x)\p_tr_{2,h}(x,t)\overline{r}_{1,h}(x,t)\dv\dd t\cr
+\int_0^T\!\!\!\int_\M q(x)u_2(x,t)\overline{u}_1(x,t)\dv dt \cr
= \int_0^T\!\!\!\int_{\p \M} \overline{g}_h \para{\Lambda_{a_1,q_1}-\Lambda_{a_2,q_2}}f_{h} \ds \, \dd t+\mathscr{R}_h.
\end{multline}
In view of (\ref{4.7}) and (\ref{4.4}), we have
\begin{equation}\label{4.11}
\abs{\mathscr{R}_h}\leq C\norm{\theta_1}_{\mathcal{V}(Q)}\norm{\theta_2}_{\mathcal{V}(Q)}.
\end{equation}
On the other hand, by the trace theorem, we find
\begin{eqnarray}\label{4.12}
\bigg|\int_0^T\!\!\!\int_{\p \M}\para{\Lambda_{a_1,q_1}-\Lambda_{a_2,q_2}}(f_{h}) \overline{g}_h \ds \, \dd t \bigg|
&\leq & \norm{\Lambda_{a_1,q_1}-\Lambda_{a_2,q_2}} \norm{f_h}_{H^{1}(\Sigma)}\norm{g_h}_{L^2(\Sigma)}\cr
&\leq & Ch^{-3}\norm{\theta_1}_{\mathcal{V}(Q)}\norm{\theta_2}_{\mathcal{V}(Q)}\norm{\Lambda_{a_1,q_1}
-\Lambda_{a_2,q_2}}.
\end{eqnarray}
The estimate (\ref{4.2}) follows easily from (\ref{4.10}), (\ref{4.11}) and (\ref{4.12}).\\
This completes the proof of the Lemma.
\end{proof}
\begin{lemma}\label{L.4.2} There exists $C>0$ such that for any $\Psi\in H^2 (S_y\M_{1})$, the following estimate
\begin{multline}\label{4.13}
\abs{\int_{S^+_{y}\M_1}\para{\exp\para{-\frac{1}{2}\I(a)(y,\xi)}-1}
\Psi(\xi)  \,  
\dd\omega_y(\xi)}\cr
 \leq C\para{h+h^{-2}\norm{\Lambda_{a_1,q_1}-\Lambda_{a_2,q_2}}} \norm{\Psi}_{H^2(S_y\M_{1})}.
\end{multline}
holds for any $y\in\p\M_1$.
\end{lemma}
We use the notation
    $$ S_y^+\M_{1} = \big\{\xi \in S_{y}\M_{1} : \langle \nu,\xi \rangle<0 \big\}. $$
\begin{proof}
We take two solutions to (\ref{3.11}) of the form
\begin{align*}
\widetilde{\theta}_1(t,r,\xi)&=\alpha^{-1/4}\phi(t-r)\Psi(\xi), \\
\widetilde{\theta}_2(t,r,\xi)&=\alpha^{-1/4}\phi(t-r).
\end{align*}
Now we change variable in the left term of (\ref{4.1}), $x=\exp_{y}(r\xi)$, $r>0$ and
$\xi\in S_{y}\M_1$, we have
\begin{multline}\label{4.14}
\int_0^T\!\!\int_\M a(x) (\overline{\theta}_1\theta_2)(x,t)\psi_{a}(x,t)  \dv \, \dd t \cr
=\int_0^T\!\!\int_{S^+_{y}\M_1}\!\int_0^{\ell_+(y,\xi)}\widetilde{a}(r,y,\xi)(\overline{\widetilde{\theta}}_1\widetilde{\theta}_2)( t,r,\xi)
\widetilde{\psi}_{a}( t,r,\xi)
\alpha^{1/2} \, \dd r \, \dd\omega_y(\xi) \, \dd t\cr
=\int_0^T\!\!\int_{S^+_{y}\M_1}\!\int_0^{\ell_+(y,\xi)}\widetilde{a}(r,y,\xi)\phi^2( t-r)\widetilde{\psi}_{a}( t,r,\xi)\Psi(\xi)  \, \dd r \,
\dd\omega_y(\xi)
\, \dd t\cr
=\int_0^T\!\!\int_{S^+_{y}\M_1}\!\int_\R \widetilde{a}( t-\tau,y,\xi)\phi^2(\tau)\widetilde{\psi}_{a}( t, t-\tau,\xi)\Psi(\xi) \, \dd \tau \,
\dd\omega_y(\xi)
\, \dd t\cr
=\int_0^T\!\!\int_{S^+_{y}\M_1}\!\int_\R \widetilde{a}( t-\tau,y,\xi)\phi^2(\tau)\exp\para{-\frac{1}{2}\int_0^{ t}\widetilde{a}(s-\tau,y,\xi)ds}
\Psi(\xi)  \, \dd \tau \,
\dd\omega_y(\xi)\cr
=2\int_\R\phi^2(\tau) \!\!\int_{S^+_{y}\M_1}\int_0^T \frac{d}{dt}\exp\para{-\frac{1}{2}\int_0^{ t}\widetilde{a}(s-\tau,y,\xi)ds}
\Psi(\xi)  \, \dd \tau \,
\dd\omega_y(\xi)\cr
=2\int_\R\phi^2(\tau) \!\!\int_{S^+_{y}\M_1}\cro{\exp\para{-\frac{1}{2}\int_0^{ T}\widetilde{a}(s-\tau,y,\xi)ds}-1}
\Psi(\xi)  \, \dd \tau \,
\dd\omega_y(\xi).
\end{multline}
By the support properties of the function $\phi$, we get that the left-hand side term in (\ref{4.14}) reads
\begin{multline*}
  \int_\R\phi^2(\tau) \!\!\int_{S^+_{y}\M_1}\cro{\exp\para{-\frac{1}{2}\int_0^{T}\widetilde{a}(s-\tau,y,\xi)ds}-1}
\Psi(\xi)  \, \dd \tau \,
\dd\omega_y(\xi)=\cr
\int_{S^+_{y}\M_1}\cro{\exp\para{-\frac{1}{2}\int_0^{\ell_+(y,\xi)}\widetilde{a}(s,y,\xi)ds}-1}
\Psi(\xi) \mu(y,\xi) \, 
\dd\omega_y(\xi).
\end{multline*}
Then, by (\ref{4.14}) and (\ref{4.2}) we get
\begin{multline}\label{4.15}
\abs{\int_{S^+_{y}\M_1}\para{\exp\para{-\frac{1}{2}\I(a)(y,\xi)}-1}
\Psi(\xi)  \,  
\dd\omega_y(\xi)}\cr
 \leq C\para{h+h^{-2}\norm{\Lambda_{a_1,q_1}-\Lambda_{a_2,q_2}}} \norm{\Psi}_{H^2(S_y\M_{1})}.
\end{multline}
This completes the proof of the Lemma.
\end{proof}
\subsection{End of the proof of the stability estimate of the absorption coefficient}
Let us now complete the proof of the stability estimate of the absorption coefficient. 
\smallskip

We define the Poisson kernel of $B(0,1)\subset T_y \M_1$, i.e.,
$$
P(\theta,\xi)=\frac{1-\abs{\theta}^2}{\alpha_n\abs{\theta-\xi}^n},\quad \theta\in B(0,1);\,\, \xi\in S_y\M_1.
$$
For $0<\kappa<1$, we define $\Psi_\kappa: S_y\M_1\times S_y\M_1\to \R$ as
\begin{align}\label{4.16}
\Psi_\kappa(\theta,\xi)=P(\kappa\theta,\xi).
\end{align}
We have the following Lemma (see Appendix A for the proof).
\begin{lemma}\label{L.4.3}
Let $\Psi_\kappa$ given by (\ref{4.16}), $\kappa \in (0,1)$. Then we have the following properties:
\begin{equation}\label{4.17}
\displaystyle 0\leq\Psi_\kappa(\theta,\xi)\leq \frac{2}{\alpha_n(1-\kappa)^{n-1}},\quad \forall\, \kappa\in (0,1),\,\forall\,\xi,\theta \in S_y\M_1.
\end{equation}  
\begin{equation}\label{4.18}
 \displaystyle \int_{S_y\M_1}\Psi_\kappa(\theta,\xi) d\omega_y(\xi)=1,\quad \forall \kappa\in (0,1),\,\forall\,\theta \in S_y\M_1.
\end{equation} 
\begin{equation}\label{4.19}
\int_{S_y\M_1}\Psi_\kappa(\theta,\xi) \abs{\theta-\xi}d\omega_y(\xi)\leq C(1-\kappa)^{1/2n}, \quad \forall\, \kappa\in (0,1),\,\forall\,\theta \in S_y\M_1.
\end{equation}
\begin{equation}\label{4.20}
\norm{ \Psi_\kappa(\theta,\cdot)}_{H^2(S_y\M_1)}^2\leq \frac{C}{(1-\kappa)^{n+3}}, \quad \forall \kappa\in (0,1),\,\forall\,\theta \in S_y\M_1.
\end{equation}
\end{lemma}
\begin{lemma}\label{L.4.4}
 Let $a_{i}\in \mathscr{A}(m_{1},\alpha)$, $q_i \in \mathscr{Q}(m_2)$, $i=1,\,2$. There exist $C>0,$
$\delta>0$, $\beta>0$ and $h_{0}>0$ such that 
\begin{equation}\label{4.21}
|\I(a)(y,\theta)|\leq C \Big( h^{-\delta}\|\Lambda_{a_{2},q_{2}}-\Lambda_{a_{1},q_{1}}\|+h^{\beta}  \Big),\quad\forall \, (y,\theta)\in\p_+S\M_1,
\end{equation}
for any $h\leq h_{0}$. Here $C$ depends only on $\M$, $T$,
$m_{1}$ and $m_{2}$.
\end{lemma}
\begin{proof}
Let $(y,\theta)\in\p_+S\M_1$ be a fixed and let $\Psi_\kappa$ be the positive function given by (\ref{4.16}). We extend $\I(a)$ by zero in $\p_-S\M$, then,  we have
\begin{multline}\label{4.22}
\Big|\exp\para{-\displaystyle\frac{1}{2}\,\I(a)(y,\theta)} -1 \Big|=\Big|
\displaystyle\int_{S_y\M_1}\Psi_{\kappa}(\theta,\xi)\Big[ \exp\Big(-\frac{1}{2}\, \I(a)(y,\theta) \Big)-1  \Big]\,d\omega_y(\xi)\Big|\cr
\leq\quad \Big| \displaystyle\int_{S_y\M_1}\Psi_{\kappa}(\theta,\xi)\Big[\exp\Big( -\frac{1}{2}\,\I(a)(y,\theta) \Big)
-\exp\Big(-\displaystyle\frac{1}{2}\,\I(a)(y,\xi)\Big)   \Big]  \,d\omega_y(\xi)\Big|\cr
 +\Big| \displaystyle\int_{S_y\M_1} \Psi_{\kappa}(\theta,\xi)\Big[\exp\Big(-\displaystyle\frac{1}{2}\,\I(a)(y,\xi)  \Big)-1 \Big]
 d\omega_y(\xi)\Big|.
\end{multline}
Therefore, since we have
$$\begin{array}{lll} \Big| \!\exp\Big(
\!-\displaystyle\frac{1}{2}\,\I(a)(y,\theta)\Big)
\!-\!\exp\Big(\! -\displaystyle\frac{1}{2}\,\I(a)(y,\xi)
\! \Big) \Big|\!\!\!&\leq&\!\!\! C\Big|\displaystyle \I(a)(y,\theta)-\I(a)(y,\xi)\Big|,\cr
\end{array}$$
and using the fact that
$$
\Big|\displaystyle \I(a)(y,\theta)-\I(a)(y,\xi)\Big|\leq C \,|\theta-\xi|,
$$
 we deduce upon applying Lemma \ref{L.4.2} with $\Psi=\Psi_{\kappa}(\theta,\cdot)$ the following estimation
\begin{multline*}
\Big|\exp\Big(-\frac{1}{2}\,\I(a)(y,\theta)\Big)-1
\Big|\leq \cr
C\int_{S_y\M_1}\Psi_{\kappa}(\theta,\xi)\,|\theta-\xi|\,d\omega_y(\xi)+C\Big(
h^{-2}\|\Lambda_{a_{2},q_{2}}-\Lambda_{a_{1},q_{1}}\|+h
\Big)\|\Psi_{\kappa}\|_{H^2(S_y\M_1)}^{2}.
\end{multline*}
 On the other hand, by (\ref{4.20}) and (\ref{4.19}), we have the following inequality
$$
 \Big|\exp\Big(-\frac{1}{2}\, \I(a)(y,\theta)\Big)-1
\Big|\leq C \,(1-\kappa)^{1/2n}+C\Big(
h^{-2}\|\Lambda_{a_{2},q_{2}}-\Lambda_{a_{1},q_{1}}\|+h
\Big)(1-\kappa)^{-(3+n)}.
 $$
 Selecting $(1-\kappa)$ small such that $(1-\kappa)^{1/2n}=h(1-\kappa)^{-(n+3)}$,
that is $(1-\kappa)=h^{2n/1+2n^2+6n}$, we find two constants $\delta>0$ and $\beta>0$
such that
$$\Big| \exp \Big( -\frac{1}{2}\, \I(a)(y,\theta) \Big)-1 \Big|\leq
C \Big[
h^{-\delta}\|\Lambda_{a_{2},q_{2}}-\Lambda_{a_{1},q_{1}}\|+h^{\beta}
\Big].$$
 Now, using the fact that $|X|\leq e^{M}\,|e^{X}-1|$ for any $|X|\leq M$, we
deduce that
$$\Big| -\frac{1}{2}\, \I(a)(y,\theta) \Big|\leq e^{M_{1}T}\Big| \exp\Big( -\frac{1}{2}\, \I(a)(y,\theta)  \Big)-1
\Big|.$$ Hence, we conclude that for all $\theta\in S_y\M_1$ and
$y\in $ we have
$$
\Big|\I(a)(y,\theta) \Big|\leq C \Big(h^{-\delta}\|\Lambda_{a_{2},q_{2}}-\Lambda_{a_{1},q_{1}}\|+h^{\beta}\Big).
$$ 
The proof of Lemma \ref{L.4.4} is complete.
\end{proof}
Integrating the estimate (\ref{4.21}) over $\p_+S\M_1,$ with respect to $ \mu(y,\theta) \,  
\dss (y,\theta)$, then minimizing on $h$, we get 
\begin{equation}\label{4.23}
\|\I(a)\|_{L^2(\p_+S\M_1)}\leq C \|\Lambda_{a_{2},q_{2}}-\Lambda_{a_{1},q_{1}}\|^{\frac{\beta}{\beta+\delta}}.
\end{equation}
With respect to the Theorem \ref{theorem}, we have 
\begin{equation}\label{4.24}
\|a\|_{L^2(\M)}\leq C \|\I(a)\|_{H^1(\p_+S\M_1)}.
\end{equation}
By  interpolation inequality and (\ref{2.7}), we have
\begin{align}\label{4.25}
\|\I(a)\|^2_{H^1(\p_+S\M_1)} & \leq C \|\I(a)\|_{L^2(\p_+S\M_1)}\|\I(a)\|_{H^2(\p_+S\M_1)} \cr &  \leq C \|\I(a)\|_{L^2(\p_+S\M_1)}.
\end{align}
Combining with (\ref{4.24}) and (\ref{4.25}), we get 
$$
\|a\|_{L^2(\M)}\leq C \|\Lambda_{a_{2},q_{2}}-\Lambda_{a_{1},q_{1}}\|^{s_0},
$$
where $s_0=\frac{\beta}{2(\beta+\delta)}$.\\ 
Moreover, let $\eta_0\in (n/2,\eta)$, by Sobolev emedding and interpolation inequality, there exists $\delta\in(0,1)$ such that 
\begin{align}\label{4.26}
\|a\|_{\mathcal{C}^0(\M)}\leq\norm{a}_{H^{\eta_0}(\M)}\leq C \|a\|^{\delta}_{L^2(\M)}\|a\|^{1-\delta}_{H^\eta(\M)}\leq C \|a\|^\delta_{L^2(\M)}\leq C \|\Lambda_{a_{2},q_{2}}-\Lambda_{a_{1},q_{1}}\|^{\delta s_0}.
\end{align}
 This completes the proof of the H\"older stability estimate of the absorption coefficient.
\section{Stable determination of the electric potential}
\setcounter{equation}{0}
In this section, we prove a stability estimate for the  electric potential $q$. We use the stability result obtained for  the absorption coefficient $a$. Like in the previous section,  we let $a_1,\,a_2\in\mathscr{A}(m_1,\eta)$ and $q_1,q_2\in\mathscr{Q}(m_2)$ such that $a_1=a_2$, $q_1=q_2$ near  the boundary $\p\M$. We set
$$
a(x)=(a_1-a_2)(x),\quad q(x)=(q_1-q_2)(x).
$$
Recall that we have extended $a_{1},a_{2}$ as $W^{2,\infty}(\M_1)$ in such a way that $a=0$ and $q=0$ on $\M_{1} \setminus \M$.\\
We denote $\psi_{a_2}\in\mathcal{W}(Q)$ and $\psi_{-a_1}\in \mathcal{W}(Q)$ the solutions of (\ref{3.12}) respectively with $a=a_2$ and $a=-a_1$ given by (\ref{3.22}), and set
\begin{equation}\label{5.1}
\psi_a(x,t)=\psi_{a_2}(x,t)\psi_{-a_1}(x,t).
\end{equation} 
We have a preleminary estimate.
\begin{lemma}\label{L.5.1}
Let $T>0$. There exist $C>0$ such that for any $\theta_j\in \mathcal{V}(Q)$, $j=1,2$, 
satisfying the transport equation (\ref{3.11}), the following estimate holds true:
\begin{equation}\label{5.2}
\abs{\int_{0}^T\!\!\!\!\int_{\M}q(x)(\theta_2\overline{\theta}_1)(x,t)\,\dv \, \dd t } 
\leq C\para{h+h^{-1} \|a\|_{\mathcal{C}^0(\M)}+h^{-3}\norm{\Lambda_{a_1,q_1}-\Lambda_{a_2,q_2}}}\norm{\theta_1}_{\mathcal{V}(Q)}\norm{\theta_2}_{\mathcal{V}(Q)}
\end{equation}
for all $h\in (0,h_0)$.
\end{lemma}
\begin{proof}
From the equality (\ref{4.10}) and the expressions (\ref{4.6}) and (\ref{4.3}), it follows  that
\begin{multline}\label{5.3}
\int_0^T\!\!\!\int_\M q(x)(\theta_2\overline{\theta}_1)(x,t)\psi_{a}(x,t)\,\dv dt = -
\int_0^T\!\!\!\int_{\p \M} \overline{g}_h\para{\Lambda_{a_1,q_1}-\Lambda_{a_2,q_2}}f_{h} \ds \, \dd t
\cr +ih^{-1}\int_0^T\!\!\!\int_\M a(x) (\theta_2 \overline{\theta}_1)(x,t)(\psi_{a_2}\psi_{-a_1})( x,t)\dv\,\dd t +ih^{-1}\int_0^T\!\!\!\int_\M a(x) (\theta_2\psi_{a_2})(x,t))\overline{r}_{1,h}e^{i\varphi/h}\dv\dd t\cr
-\int_0^T\!\!\!\int_\M a(x)\p_t(\theta_2\psi_{a_2})(x,t)(\overline{\theta}_1\psi_{a_1})(x,t)\dv\dd t
-\int_0^T\!\!\!\int_\M a(x)\p_t(\theta_2\psi_{a_2})(x,t)\overline{r}_{1,h}(x,t)e^{i\varphi/h}\dv\dd t\cr
-\int_0^T\!\!\!\int_\M a(x)\p_t r_{2,h}(\overline{\theta}_1\psi_{a_1})(x,t) e^{-i\varphi/h}\dv\dd t
-\int_0^T\!\!\!\int_\M a(x)\p_tr_{2,h}(x,t)\overline{r}_{1,h}(x,t)\dv\dd t\cr
-\int_0^T\!\!\!\int_\M q(x)(\overline{\theta}_1\psi_{-a_1})(x,t)r_{2,h}(x,t)e^{-i\varphi/h}\dv\dd t -\int_0^T\!\!\!\int_\M q(x)(\theta_2\psi_{a_2})(x,t)\overline{r}_{1,h}(x,t)e^{i\varphi/h}\dv\dd t\cr -\int_0^T\!\!\!\int_\M q(x)(r_{2,h}\overline{r}_{1,h})(x,t)\dv\dd t .
\end{multline}
We set 
$$\int_0^T\!\!\!\int_\M q(x)(\theta_2\overline{\theta}_1)(x,t)\psi_{a}(x,t)\,\dv dt = -\int_0^T\!\!\!\int_{\p \M} \overline{g}_h\para{\Lambda_{a_1,q_1}-\Lambda_{a_2,q_2}}f_{h} \ds \, \dd t +\mathscr{R}'_h.$$
In view of (\ref{4.7}) and (\ref{4.4}), we have
\begin{equation}\label{5.4}
\abs{\mathscr{R}'_h}\leq C\para{h+h^{-1} \|a\|_{\mathcal{C}^0(\M)}}\norm{\theta_1}_{\mathcal{V}(Q)}\norm{\theta_2}_{\mathcal{V}(Q)}.
\end{equation}
On the other hand, by the trace theorem, we find
\begin{eqnarray}\label{5.5}
\bigg|\int_0^T\!\!\!\int_{\p \M}\para{\Lambda_{a_1,q_1}-\Lambda_{a_2,q_2}}(f_{h}) \overline{g}_h \ds \, \dd t \bigg|
&\leq & \norm{\Lambda_{a_1,q_1}-\Lambda_{a_2,q_2}} \norm{f_h}_{H^{1}(\Sigma)}\norm{g_h}_{L^2(\Sigma)}\cr
&\leq & Ch^{-3}\norm{\theta_1}_{\mathcal{V}(Q)}\norm{\theta_2}_{\mathcal{V}(Q)}\norm{\Lambda_{a_1,q_1}
-\Lambda_{a_2,q_2}}.
\end{eqnarray}
Furthermore
\begin{multline} \label{5.6}
\int_0^T\!\!\!\int_\M q(x)(\theta_2\overline{\theta}_1)(x,t)\,\dv dt = \int_0^T\!\!\!\int_\M q(x)(\theta_2\overline{\theta}_1)(x,t)(1-\psi_{a}(x,t))\,\dv dt \cr + \int_0^T\!\!\!\int_\M q(x)(\theta_2\overline{\theta}_1)(x,t)\psi_{a}(x,t)\,\dv dt
\end{multline}
and since
$$
\abs{1-\psi_a(x,t)}\leq C\norm{a}_{\mathcal{C}(\M)}
$$
we deduce that 
 $$
 \bigg|\int_0^T\!\!\!\int_\M q(x)(\theta_2\overline{\theta}_1)(x,t)(1-\psi_{a}(x,t))\,\dv dt\bigg| \leq C\norm{a}_{\mathcal{C}^0(\M)} \norm{\theta_1}_{\mathcal{V}(Q)}\norm{\theta_2}_{\mathcal{V}(Q)}.$$ Combining with (\ref{5.4}), (\ref{5.5}) and (\ref{5.6}), the estimate (\ref{5.2}) follows.\\
  This completes the proof of the Lemma.
\end{proof}
\begin{lemma}\label{L.5.2} There exists $C>0$ such that for any $\Psi\in H^2 (S_y\M_{1})$, the following estimate
\begin{equation}\label{5.7}
\bigg|\int_{S_{y}\M_1}\I q(y,\xi)
\Psi(\xi)  \,  \,
\dd\omega_y(\xi)\bigg|
\leq C\para{h+h^{-1} \norm{a}_{\mathcal{C}^0(\M)}+h^{-3}\norm{\Lambda_{a_1,q_1}
-\Lambda_{a_2,q_2}}
} \norm{\Psi}_{H^2(S^+_y\M_{1})},
\end{equation}
holds for any $y\in\p\M_1$.
\end{lemma}
\begin{proof}
As in Lemma \ref{L.4.2}, we take two solutions to (\ref{3.11}) of the form
\begin{align*}
\widetilde{\theta}_1(t,r,\xi)&=\alpha^{-1/4}\phi(t-r)\Psi(\xi), \\
\widetilde{\theta}_2(t,r,\xi)&=\alpha^{-1/4}\phi(t-r).
\end{align*}
We change variable in the left term of (\ref{5.2}), $x=\exp_{y}(r\xi)$, $r>0$ and $\xi\in S_{y}\M_1$. We have
\begin{multline}\label{5.8}
\int_0^T\!\!\int_\M q(x) (\overline{\theta}_1\theta_2)(x,t) \dv \, \dd t \cr
=\int_0^T\!\!\int_{S_{y}\M_1}\!\int_0^{\ell_+(y,\xi)}\widetilde{q}(r,y,\xi)(\overline{\widetilde{\theta}}_1\widetilde{\theta}_2)( t,r,\xi)
\alpha^{1/2} \, \dd r \, \dd\omega_y(\xi) \, \dd t\cr
=\int_0^T\!\!\int_{S_{y}\M_1}\!\int_0^{\ell_+(y,\xi)}\widetilde{q}(r,y,\xi)\phi^2( t-r)\Psi(\xi)  \, \dd r \,
\dd\omega_y(\xi)
\, \dd t.
\end{multline}
By support properties of the function $\phi$, we get
$$\int_0^T\!\!\int_\M q(x) (\overline{\theta}_1\theta_2)(x,t) \dv \, \dd t=\bigg(\int_\R\phi^2(t)\,dt \bigg) \!\!\int_{S_{y}\M_1}\I q(y,\xi)
\Psi(\xi)  \,  \,
\dd\omega_y(\xi).
$$
Then, by (\ref{5.2}), we get
\begin{equation}\label{5.9}
\bigg|\int_{S_{y}\M_1}\I q(y,\xi)
\Psi(\xi)  \,  \,
\dd\omega_y(\xi)\bigg|
\leq C\para{h+h^{-1} \|a\|_{\mathcal{C}^0(\M)}+h^{-3}\norm{\Lambda_{a_1,q_1}
-\Lambda_{a_2,q_2}}
} \norm{\Psi}_{H^2(S^+_y\M_{1})}.
\end{equation}
This compete the proof.
\end{proof}
Let us now complete the proof of the stability estimate of the electric potentiel. We recall the definition of the Poisson kernel of $B(0,1)\subset T_y \M_1$, i-e.,
$$
P(\theta,\xi)=\frac{1-\abs{\theta}^2}{\alpha_n\abs{\theta-\xi}^n},\quad \theta\in B(0,1);\,\, \xi\in S_y\M_1.
$$
For $0<\kappa<1$, we define $\Psi_\kappa: S_y\M_1\times S_y\M_1\to \R$ as
\begin{align}\label{5.10}
\Psi_\kappa(\theta,\xi)=P(\kappa\theta,\xi).
\end{align}
\begin{lemma}\label{L.5.3}
 Let $a_{i}\in \mathscr{A}(m_1,\alpha)$, $q_i\in \mathscr{Q}(m_2)$, $i=1,\,2$. There exist $C>0,$
$\delta>0$, $\beta>0$ and $h_{0}>0$ such that 
\begin{equation}\label{5.11}
|\I(q)(y,\theta)|\leq C \Big( h^{-\delta}\|\Lambda_{a_{2},b_{2}}-\Lambda_{a_{1},b_{1}}\|+ h^{-\delta}\|a\|_{\mathcal{C}^0(\M)}+h^{\beta}  \Big),\quad\forall \, (y,\theta)\in\p_+S\M_1,
\end{equation}
for any $h\leq h_{0}$. Here $C$ depends only on $\M$, $T$,
$m_{1}$ and $m_{2}$.
\end{lemma}
\begin{proof}
We fix $(y,\theta)\in\p_+S\M$ and let the positive function $\Psi_\kappa$ given by (\ref{5.10}),  we have
\begin{multline}\label{5.12}
\Big|\I(q)(y,\theta)\Big|=\Big|
\displaystyle\int_{S_y\M_1}\Psi_{\kappa}(\theta,\xi)\I(q)(y,\theta)\,d\omega_y(\xi)\Big|\cr
\leq\quad \Big| \displaystyle\int_{S_y\M_1}\Psi_{\kappa}(\theta,\xi)
\Big(\I(q)(y,\theta)
-\I(q)(y,\xi)\Big)  \,d\omega_y(\xi)\Big|\cr
 +\Big| \displaystyle\int_{S_y\M_1} \Psi_{\kappa}(\theta,\xi)\I(q)(y,\xi)
 d\omega_y(\xi)\Big|.
\end{multline}
Since we have
$$
\Big|\displaystyle \I(q)(y,\theta)-\I(q)(y,\xi)\Big|\leq C \,|\theta-\xi|,
$$ 
then we deduce upon applying Lemma \ref{L.5.2} with $\Psi=\Psi_{\kappa}(\theta,\cdot)$ that we have the following estimation
\begin{multline*}
\Big|\I(q)(y,\theta)\Big|\leq C\int_{S_y\M_1}\Psi_{\kappa}(\theta,\xi)\,|\theta-\xi|\,d\omega_y(\xi)\cr
+C\Big(h+ h^{-1}\|a\|_{\mathcal{C}^0(\M)}+h^{-3}\|\Lambda_{a_{2},q_{2}}-\Lambda_{a_{1},q_{1}}\|
\Big)\|\Psi_{\kappa}\|_{H^2(S_y\M_1)}^{2}.
\end{multline*}
 On the other hand, by (\ref{4.20}) and (\ref{4.19}), we obtain
$$
 \Big|\I(q)(y,\theta)\Big|\leq C \,(1-\kappa)^{1/2n}+C\Big(
h^{-3}\|\Lambda_{a_{2},q_{2}}-\Lambda_{a_{1},q_{1}}\|+ h^{-1}\|a\|_{\mathcal{C}^0(\M)}+h
\Big)(1-\kappa)^{-(3+n)}.
 $$
 Take $(1-\kappa)$ small such that $(1-\kappa)^{1/2n}=h(1-\kappa)^{-(n+3)}$,
that is $(1-\kappa)=h^{2n/1+2n^2+6n}$, we find two constants $\delta>0$ and $\beta>0$
such that
$$\Big| \I(q)(y,\theta) \Big|\leq
C \Big[
h^{-\delta}\|\Lambda_{a_{2},q_{2}}-\Lambda_{a_{1},q_{1}}\|+ h^{-\delta}\norm{a}_{\mathcal{C}^0(\M)}+h^{\beta}
\Big].$$
The proof of Lemma \ref{L.5.3} is complete.
\end{proof}
Integrating the estimate (\ref{5.11}) over $\p_+S\M_1,$ with respect to $ \mu(y,\theta) \,  
\dss (y,\theta)$, then minimizing on $h$, we get 
\begin{equation}\label{5.13}
\|\I(q)\|_{L^2(\p_+S\M_1)}\leq C ( \|\Lambda_{a_{2},q_{2}}-\Lambda_{a_{1},q_{1}}\|+\|a\|_{\mathcal{C}^0(\M)})^{\frac{\beta}{\beta+\delta}}.
\end{equation}
By interpolation inequality, we get
\begin{align}\label{5.14}
\norm{\I(q)}^2_{H^1(\p_+S\M_1)} & \leq C \|\I(q)\|_{L^2(\p_+S\M_1)}\|\I(q)\|_{H^2(\p_+S\M_1)} \cr 
&  \leq C \|\I(q)\|_{L^2(\p_+S\M_1)}.
\end{align}
From the Theorem \ref{theorem}, it follows that 
\begin{equation}\label{5.15}
\|q\|_{L^2(\M)}\leq C ( \|\Lambda_{a_{2},q_{2}}-\Lambda_{a_{1},q_{1}}\|+\|a\|_{\mathcal{C}^0(\M)})^{\frac{\beta}{\beta+\delta}}.
\end{equation}
Using the estimate (\ref{4.26}), we conclude that
$$
\|q\|_{L^2(\M)}\leq C  \|\Lambda_{a_{2},q_{2}}-\Lambda_{a_{1},q_{1}}\|^{s_1},
$$ 
where $s_1 \in (0,1)$.
This completes the proof of the Theorem \ref{Th2}.
\appendix
\section{Proof of Lemma \ref{L.4.3}}
\setcounter{equation}{0}
We define the Poisson kernel of $B(0,1)\subset T_y \M_1$, i-e.,
$$
P(\theta,\xi)=\frac{1-\abs{\theta}^2}{\alpha_n\abs{\theta-\xi}^n},\quad \theta\in B(0,1);\,\, \xi\in S_y\M_1.
$$
For $0<\kappa<1$, we define $\Psi_\kappa: S_y\M_1\times S_y\M_1\to \R$ as
$$
\Psi_\kappa(\theta,\xi)=P(\kappa\theta,\xi).
$$
Let $P_0$ the Poisson kernel for the Euclidian unit ball $B_0(0,1)\subset\R^n$ i.e., 
$$
P_0(\hat{\theta},\hat{\xi})=\frac{1-\abs{\hat{\theta}}_0^2}{\alpha_n\abs{\hat{\theta}-\hat{\xi}}_0^n},\quad \hat{\theta}\in B_0(0,1);\,\, \hat{\xi}\in \mathbb{S}^{n-1}.
$$
where $\abs{\cdot}_0$ is the Euclidian norm of $\R^n$. From the well known properties of $P_0$, we have 
$$
\int_{\mathbb{S}^{n-1}}P_0(\kappa\hat{\theta},\hat{\xi}) d\omega_0(\hat{\xi})=1, \quad\textrm{for all}\, \kappa\in (0,1),\,\, \hat{\theta}\in \mathbb{S}^{n-1}.
$$
Let $\gamma=\sqrt{\g}:=(\gamma_{ij})$ be definite symetric positive matrix such that $\gamma^2=(g_{ij})$, then we get
$$
P(\theta,\xi)=\frac{1-\abs{\gamma^{-1}\theta}_0^2}{\alpha_n\abs{\gamma^{-1}\theta-\gamma^{-1}\xi}_0^n},\quad \theta,\xi\in S_y\M.
$$
we deduce from the change of variable $\hat{\xi}=\gamma^{-1}\xi$
\begin{eqnarray}\label{A1}
\int_{S_y\M_1}P(\kappa\theta,\xi)d\omega_y(\xi)&=&\int_{S_y\M_1}P_0(\kappa\gamma^{-1}\theta,\gamma^{-1}\xi)d\omega_y(\xi)\cr
&=&\frac{1}{\det\gamma}\int_{\mathbb{S}^{n-1}}P_0(\kappa\gamma^{-1}\theta,\hat{\xi})(\det\gamma) d\omega_0(\hat{\xi})=1
\end{eqnarray}
This complete the proof of (\ref{4.18}).
Let now
$$
V_\theta=\set{\xi\in S_y\M_1,\,\,\abs{\kappa\theta-\xi}\leq (1-\kappa)^{1/2n}}.
$$
\begin{multline}\label{A.2}
\int_{S_y\M_1}\Psi_\kappa(\theta,\xi) \abs{\theta-\xi}d\omega_y(\xi)\leq\int_{S_y\M_1}\Psi_\kappa(\theta,\xi) \abs{\kappa\theta-\xi}d\omega_y(\xi)+(1-\kappa)\cr
\int_{V_\theta}\Psi_\kappa(\theta,\xi) \abs{\kappa\theta-\xi}d\omega_y(\xi)+\int_{S_y\M_1\backslash V_\theta}\Psi_\kappa(\theta,\xi) \abs{\kappa\theta-\xi}d\omega_y(\xi)+(1-\kappa)\cr
(1-\kappa)^{1/2n}+4(1-\kappa)^{1/2}+(1-\kappa)\leq C(1-\kappa)^{1/2n},
\end{multline}
and then we get (\ref{4.19}).\\
By a simple computation, we have
$$
\abs{\nabla_\xi^k P(\kappa\theta,\xi)}\leq C\frac{1-\kappa^2}{\abs{\kappa\theta-\xi}^{n+k}},\quad k=1,2,
$$
we deduce that
\begin{equation}
\norm{ \Psi_\kappa(\theta,\cdot)}^2_{H^2(S_y\M_1)} \leq C\frac{1-\kappa^2}{(1-\kappa)^{n+4}}\int_{S_y\M_1}\frac{1-\kappa^2}{\abs{\kappa\theta-\xi}^n} d\omega_y(\xi)\leq 
\frac{C}{(1-\kappa)^{n+3}}.
\end{equation}
This complete the proof of (\ref{4.20}).

\end{document}